\documentclass[a4paper,11pt]{amsart}
\usepackage[T1]{fontenc} 
\usepackage{amsmath,amsthm}
\usepackage{enumerate}
\usepackage{mathtools}
\usepackage[francais,english]{babel}
\usepackage[dvips]{graphicx}
\usepackage{todonotes}
\usepackage{mathdots} 
\usepackage{color}
\usepackage{tikz}
\usepackage{tikz-cd}
\usetikzlibrary{arrows,positioning,shapes,matrix,backgrounds}
\usetikzlibrary{fit,calc,decorations.pathreplacing}

\usepgflibrary{arrows}
\usepackage{soul}
\usepackage{amsfonts,amssymb}
\usepackage{geometry}
\usepackage{hyperref}
\usepackage{stmaryrd}
\usepackage{fancyhdr}
\usepackage{url}
\usepackage{todonotes}
\usepackage{dsfont}
\usepackage[utf8]{inputenc} 
\usepackage{enumerate}
\usepackage[shortlabels]{enumitem}
\usepackage[all]{xy}
\usepackage{mathtools}
\usepackage{tikz}
\usepackage{extarrows}

\theoremstyle{theorem}

\newtheorem*{thm*}{Theorem}

\newtheorem{prop}{Proposition}[section]

\newtheorem{lemma}[prop]{Lemma}

\newtheorem{thm}[prop]{Theorem}
\newtheorem{definition}[prop]{Definition}

\newtheorem{corollary}[prop]{Corollary}

\newtheorem{obs}[prop]{Observation}

\newtheoremstyle{pourlesremarques}{\topsep}{\topsep}{\normalfont}{}{\bfseries}{.}{ }{}
\theoremstyle{pourlesremarques}
\newtheorem{rem}[prop]{Remark}
\newtheorem*{rem*}{Remark}
\newtheoremstyle{pourlesexemples}{\topsep}{\topsep}{\normalfont}{}{\bfseries}{.}{ }{}
\theoremstyle{pourlesexemples}

\def\presuper#1#2%
  {\mathop{}%
   \mathopen{\vphantom{#2}}^{#1}%
   \kern-\scriptspace%
   #2}

\def\Rep{\operatorname{Rep}}

\newcommand{\J}{\mathrm{J}}

\newcommand{\cP}{\mathcal{P}}

\newcommand{\Fam}{\mathcal{F}}

\newcommand{\univ}{\mathcal{U}}

\newcommand{\K}{\mathcal{K}}
\newcommand{\Kbar}{\overline{\mathcal{K}}}

\newcommand{\Hom}{\operatorname{Hom}}
\renewcommand{\subset}{\subseteq}
\renewcommand{\supset}{\supseteq}

\newcommand{\cS}{\mathcal{S}}

\newcommand{\cB}{\mathcal{B}}

\newcommand{\diag}{\operatorname{diag}}

\newcommand{\GL}{\operatorname{GL}}

\def\Spec{\operatorname{Spec}}

\newcommand{\fS}{{\mathfrak{S}}}

\newcommand{\C}{\mathbb{C}}

\newcommand{\fri}{\mathfrak{i}}

\newcommand{\cL}{\mathcal{L}}

\newcommand{\cZ}{\mathcal{Z}}
\newcommand{\cR}{\mathcal{R}}
\newcommand{\ZZ}{\mathbb{Z}}

\newcommand{\fR}{\mathfrak{R}}
\newcommand{\fI}{\mathfrak{I}}

\def\End{\operatorname{End}}

\def\GL{\operatorname{GL}}

\def\Fr{\operatorname{Fr}}

\def\qlb{\overline{\mathbb{Q}_{\ell}}}

\def\Rep{\operatorname{Rep}}

\def\Fl{\overline{\mathbb{F}_{\ell}}}
\def\Gal{\operatorname{Gal}}

\def\supp{\operatorname{supp}}

\def\diag{\operatorname{diag}}
\def\sig{\operatorname{sig}}

\def\leq{\leqslant}
\def\geq{\geqslant}

\def\wflb{W(\overline{\mathbb{F}_{\ell}})}
\def\univ{\operatorname{univ}}
\def\Frac{\operatorname{Frac}}

\subjclass{11S23, 22E50}

\title{The universal Harish--Chandra $j$-function}
\author{Gil Moss}
\address{Department of Mathematics and Statistics, University of Maine, Orono, ME 04469}
\email{gilbert.moss@maine.edu}
\author{Justin Trias}
\address{School of Mathematics, University of East Anglia, Norwich, Norfolk, NR4 7TJ, UK}
\email{j.trias@uea.ac.uk}
\date{}

\begin{document}

\maketitle

\begin{abstract}
Let $F$ be a nonarchimedean local field with residue field of cardinality $q$, let $G$ be the $F$-points of a connected reductive group defined over $F$, let $P$ and $Q$ be two parabolic subgroups with the same Levi factor $M$. We construct intertwining operators $J_{Q|P}$ and the Harish-Chandra $j$-function $j^G$ for finitely generated smooth $A[M]$-modules, where $A$ is any commutative Noetherian algebra over $\ZZ' :=\ZZ[\sqrt{q}^{-1}]$. The construction is functorial, compatible with extension of scalars, and generalizes the previous constructions in \cite{waldspurger, dat, dhkm_conjecture, girsch_twisted}. We prove a generic Schur's lemma result for parabolic induction, which circumvents the need for generic irreducibility in defining $j^G$. Setting $A=\ZZ'$ and applying the construction to finitely generated projective generators produces a universal $j$-function that is a rational function with coefficients in the Bernstein center of $M$ over $\ZZ'$, and which gives the $j$-function of any object via specializing at points of the Bernstein scheme. We conclude by characterizing the local Langlands in families morphism (when it exists) for quasisplit classical groups in terms of an equality of $j$-functions.
\end{abstract}

\section{Introduction}
Let $F$ be a nonarchimedean local field with residue field of cardinality $q$. Let $G$ be the $F$-points of a connected reductive group defined over $F$, let $P=MU_P$ and $Q=MU_Q$ be two parabolic subgroups with the same Levi factor $M$. Let $A$ be a commutative Noetherian algebra over $\ZZ' :=\ZZ[\sqrt{q}^{-1}]$, and let $i_P^G(\sigma_0)$ be the normalized parabolic induction of a smooth $A[M]$-module $\sigma_0$. To compare $i_P^G(\sigma_0)$ to $i_Q^G(\sigma_0)$, one can fix a $\ZZ[1/p]$-valued Haar measure on $U_Q$ and write down an integral operator
\begin{align}\label{integraloperator}
J_{Q|P}(\sigma_0): i_P^G(\sigma_0)&\mapsto i_Q^G(\sigma_0)\\
f&\mapsto J_{Q|P}(\sigma_0)(f)(g) = \int_{U_Q/(U_Q\cap U_P)}f(ug)du\ .\nonumber
\end{align}
But this integral does not necessarily converge. The theory of intertwining operators is concerned with identifying families of $\sigma_0$ where the operator is defined. Intertwining operators have become a fundamental tool in the harmonic analysis of $p$-adic groups and the local Langlands program. When $A=\C$, Harish-Chandra developed this theory in an analytic way, relying on the manifold structure of the families of $\sigma_0$ that are composed of orbits of the discrete series under unitary unramified twisting. Bernstein, in unpublished notes, proposed a purely algebraic approach using the geometric lemma, which was refined and used by Waldspurger in \cite{waldspurger} to simplify Harish-Chandra's proof of the Plancherel formula. In \cite{dat}, Dat constructs intertwining operators over a general ring, subject to a $(P,Q)$-regularity hypothesis, which he establishes when $A$ is a field of characteristic $\ell\neq p$ and $\sigma_0$ is irreducible.  Our goal is to expand the algebraic framework of intertwining operators to arbitrary Noetherian $\ZZ'$-algebras $A$ and arbitrary smooth $A[M]$-modules.

Let $M^0$ be the subgroup generated by all compact subgroups, let $R$ be the finite free commutative $A$-algebra $A[M/M^0]$, and let $\Psi_M:M\to R^{\times}$ be the universal unramified character. Let $\sigma=\sigma_0\cdot \Psi_M$ be the universal unramified twist of $\sigma_0$, i.e., the $R$-module $\sigma_0\otimes_AR$ with $M$-action
$\sigma(m)(v\otimes 1) = \sigma_0(m)v\otimes \overline{m}$ for $v\in \sigma_0$, $m\in M$, and $\overline{m}$ the image of $m$ in $M/M^0$. 
\begin{thm}\label{introthm:signature}
Let $\sigma_0$ be a finitely generated $A[M]$ module satisfying Schur's lemma (i.e. the natural map $A\to \End_{A[M]}(\sigma_0)$ is an isomorphism). There exists a canonical injective homomorphism of $R$-modules $$\sig_{PQ}:\Hom_{R[M]}(i_P^G(\sigma),i_Q^G(\sigma))\to \End_{R[M]}(\sigma) =R.$$
\end{thm}
Via Frobenius reciprocity, Theorem~\ref{introthm:signature} is deduced from Theorem~\ref{thm:intertwining}, which says that, among the subquotients of the geometric lemma filtration of $r_Q^Gi_P^G(\sigma)$, the subquotient isomorphic to $\sigma$ is the only one supporting a nonzero map to $\sigma$. Thus any $J:r_Q^Gi_P^G(\sigma)\to\sigma$ induces an endomorphism of $\sigma$, which we call its $(\sigma,P,Q)$-signature and denote $\sig_{PQ}(J)$.

In the special case where $A$ is a field, the ring $R=A[M/M^0]$ is an integral domain, and if the image of $\sig_{PQ}$ contained a nonzero element $b$, the map $\sig_{PQ}$ would induce an isomorphism $$\Hom_{R[1/b][M]}(i_P^G(\sigma[1/b]),i_Q^G(\sigma[1/b]))\cong R[1/b].$$ In this framework, we could define the canonical rational intertwining operator $J_{Q|P}(\sigma_0)$ to be the preimage under $\sig_{PQ}$ of $1\in R$. The isomorphism identifying $\sigma$ with a subquotient of $r_Q^Gi_P^G(\sigma)$ is defined by the integral in Equation~(\ref{integraloperator}), so this definition of $J_{Q|P}(\sigma_0)$ would extend the definition in (\ref{integraloperator}) to all $i_P^G(\sigma_0\chi)$ where $b$ does not vanish at $\chi$ (this is explained in \ref{subsec:classicalinterpretation}).

To go beyond the case where $A$ is a field, we would like a theory of intertwining operators that encompasses reduction modulo $\ell$ in a robust way for all primes $\ell\neq p$. For this we need an element $b$ in the image of $\sig_{PQ}$ that is not only nonzero, but also not divisible by any $\ell$, in order to avoid losing information at the mod $\ell$ fiber when we pass to $R[1/b]$. Even better, we would like $b$ to preserve congruences in the sense that for all $\mathfrak{p}\in \Spec(A)$, the image of $b$ in $R\otimes_AA/\mathfrak{p}$ is nonzero. Our second main result guarantees this.

Let $S$ be the multiplicative subset of $R$ generated by elements of the form $$c_d\overline{m}^d + c_{d-1}\overline{m}^{d-1} + \cdots + c_1\overline{m}+c_0,\ \ m\in M\cap G^0,\ c_i\in A$$ such that both $c_d$ and $c_0$ are units in $A$.
\begin{thm}\label{introthm:image}
With the setup of Theorem~\ref{introthm:signature}, the image of $\sig_{PQ}$ contains an element of $S$.
\end{thm}
Thus after localizing we obtain a canonical isomorphism
$$\Hom_{S^{-1}R[M]}(i_P^G(S^{-1}\sigma),i_Q^G(S^{-1}\sigma))\overset{\sim}{\to} \End_{S^{-1}R[M]}(S^{-1}\sigma) =S^{-1}R,$$ and define $J_{Q|P}(\sigma_0)$ to be the preimage of $1\in R$. Note that $J_{Q|P}(\sigma_0)$ depends on the choice of Haar measure on $U_Q$ up to a constant multiple. Except for \ref{subsec:mu}, we will remain agnostic about this choice.

The amount of localization necessary to make $\sig_{PQ}$ an isomorphism depends on $\sigma_0$, $P$, and $Q$. For example, when $Q=P$, there is no integration in Equation~(\ref{integraloperator}) and $\sigma$ is a quotient of $r_Q^Gi_P^G(\sigma)$. It follows that $1\in B$ is already in the image of $\sig_{PQ}$ and nothing needs to be inverted. We immediately deduce that $i_P^G(\sigma)$ satisfies Schur's lemma; if this result was previously known we were not aware of it.
\begin{corollary}\label{introcor:schur}
Let $\sigma_0$ be a finitely generated $A[M]$-module satisfying Schur's lemma. Then $i_P^G(\sigma)$ also satisfies Schur's lemma.
\end{corollary}

When $Q\neq P$, $\sigma$ is no longer a quotient of $r_Q^Gi_P^G(\sigma)$ (for example, when $Q=\overline{P}$, $\sigma$ is at the bottom of the filtration on $r_Q^Gi_P^G(\sigma)$). In this situation we establish Theorem~\ref{introthm:image} by computing the $(\sigma,P,Q)$-signatures of maps $r_Q^Gi_P^G(\sigma)\to \sigma$ induced by certain elements in the group ring $R[A_M]$, where $A_M$ is identity component of the center of $M$.

With these results in hand, we can construct the Harish-Chandra $j$-function $$j_P^G(\sigma_0) = J_{P|\overline{P}}\circ J_{\overline{P}|P}\in \End_{S^{-1}R[G]}(i_P^G(S^{-1}\sigma)).$$ It is independent of $P$, so we denote it $j^G(\sigma_0)$. By passing through the isomorphism $$\End_{S^{-1}R[G]}(i_P^G(S^{-1}\sigma))\cong\End_{S^{-1}R[M]}(\sigma)$$ of Corollary \ref{introcor:schur} and $\End_{S^{-1}R[M]}(\sigma)\cong S^{-1}R = S^{-1}(A[M/M^0]),$ we can make sense of $j^G(\sigma_0)$ as a ``rational function'' even when $A$ is not a domain. In fact, $j^G(\sigma_0)$ lies in $S^{-1}R^G$ where $R^G := A[(M\cap G^0)/M^0]$ (Corollary~\ref{cor:G^0}). Proposition~\ref{prop:factorization_of_j} generalizes the well-known factorization property of $j^G(\sigma_0)$ into a product of $j$-functions $j^{M_{\alpha}}(\sigma_0)$, where $M$ is a maximal Levi in $M_{\alpha}$, so that $(M\cap M_{\alpha}^0)/M^0$ is a free commutative group of rank one. When $A=\C$, the Plancherel measure $\mu(\sigma_0)$ is $j^G(\sigma_0)^{-1}$, after an appropriate choice of Haar measures on $U_P$ and $U_{\overline{P}}$. 

Because $J_{Q|P}(\sigma_0)$ can be described in a precise sense by the integral in Equation~(\ref{integraloperator}) (equivalently, because of the functorial properties of the geometric lemma) it defines a natural transformation. More precisely, if we define the functor
\begin{align*}
\fri_P^G:\Rep_A^{\text{f.g.}}(M)&\to \Rep_{S^{-1}R}(M) \\
\sigma_0&\mapsto i_P^G(S^{-1}\sigma)\ ,
\end{align*}
then $\sigma_0\mapsto J_{Q|P}(\sigma_0)$ is a natural transformation $\fri_P^G\to \fri_Q^G$. Furthermore, it is compatible with extension of scalars along any ring homomorphism $A\to A'$, so it behaves well with congruences.

The Schur's lemma hypothesis in Theorems~\ref{introthm:signature} and~\ref{introthm:image} can be removed by considering the action of the Bernstein center $\cZ_{A,M}$, which is defined as the center of the category $\Rep_A(M)$ of smooth $A[M]$-modules, i.e.
$$\cZ_{A,M}:=\End(\text{id}_{\Rep_A(M)}).$$ The natural action of $\cZ_{A,M}$ endows $\sigma_0$ with the structure of a smooth $\cZ_{A,M}[M]$-module, hence if we define $\fR :=\cZ_{A,M}[M/M^0]$, we can view $\sigma$ as an $\fR[M]$-module (c.f. Remark~\ref{rem:differentaction}). The conclusions of Theorems~\ref{introthm:signature} and~\ref{introthm:image} hold with $\fR$ replacing $R$ and $\fS$ replacing $S$, where $\fS$ is the multiplicative subset of $\fR$ whose definition is analogous to that of $S$.  In fact, we prove the main theorems over $\fR$ in Sections~\ref{sec:unramifiedtwisting} and \ref{sec:proofs}, then use the Schur's lemma hypothesis to reduce them in \ref{sec:schur} to the statements above. Note that without Schur's lemma, and by extension without Corollary~\ref{introcor:schur}, we lose the description of $j^G(\sigma_0)$ as a rational function over $A$, but we show it lands in $Z(\End_{\fS^{-1}\fR[M]}(\fS^{-1}\sigma))$ (Proposition~\ref{prop:centrality}).

The upshot of working over $\cZ_{A,M}$ is that we can take $A$ to be $\ZZ'$ and $\sigma_0$ to be a finitely generated projective $\ZZ'[M]$-module, for example, a generator $\cP_{0,r}$ of the category of depth $\leq r$ for some $r$. In this situation we have $Z(\End_{\fS^{-1}\fR[M]}(\fS^{-1}\sigma))\cong \fS^{-1}\fR$, so $j^G(\cP_{0,r})$ can be viewed as a rational function with coefficients in $\cZ_{\ZZ',M}$. Since $J_{Q|P}(-)$ and $j^G(-)$ are well-behaved with respect to scalar extension along homomorphisms of $\ZZ'$-algebras $\lambda:\cZ_{\ZZ',M}\to B$, the intertwining operator $J_{Q|P}(\cP_{0,r})$ and $j$-function $j^G(\cP_{0,r})$ are universal in the sense that they give rise to all intertwining operators and $j$-functions by extending scalars along $\lambda$ and taking quotients. In turn, this provides a way to extend the construction beyond finitely generated objects. This is made precise in Section~\ref{sec:universal}.

\subsection{Relation to other work}
The algebraic constructions of rational intertwining operators in \cite{waldspurger, dat},\cite[Appendix C]{dhkm_conjecture} were restricted to finite length objects of $\Rep_k(M)$, where $k$ is a field of characteristic $\ell\neq p$ (or $k=\C$ in \cite{waldspurger}). In generalizing the doubling method for classical groups, Girsch (\cite[Prop 9.1]{girsch_twisted}) constructed intertwining operators for admissible objects in $\Rep_A^{\text{f.g.}}(M)$, where $M$ is a maximal Levi in a classical group and $A$ is an Noetherian $\mathbb{Z}[1/p]$-algebra. Girsch's strategy inspired our approach and, like ours, relies on the recent finiteness theorems in \cite{dhkm_finiteness} (which, in turn, rely on \cite{fargues_scholze}). 

In \cite{waldspurger}, Waldspurger constructs $j^G(\sigma_0)$ for irreducible $\sigma_0$ generic irreducibility. Dat uses the theory of intertwining operators to prove generic irreducibility over $k$ for $\ell>0$ in \cite{dat} (under restrictions on $G$ that were subsequently removed in \cite{dhkm_finiteness}), and then applies generic irreducibility to construct $j^G(\sigma_0)$. When $A$ is not a field, the notions of irreducibility and finite length are not well-behaved. We circumvent the need for generic irreducibility or the finite length hypothesis with the $(\sigma_0,P,Q)$-signature and the generic Schur's lemma property of Corollary~\ref{introcor:schur}. The results in \cite{dhkm_finiteness} allow us to work with finitely generated objects, and we eventually generalize to all objects. 

\subsection{An application to local Langlands for quasisplit classical groups}

As an application, we use a converse theorem for the $j$-function over $\C$ to show that a local Langlands in families morphism from the ring of functions on the moduli space of Langlands parameters to the Bernstein center is uniquely characterized (if it exists) by the property that it respects universal $j$-functions. We give an overview here and refer to Section~\ref{sec:characterization} for more details and references. 

Let $W_F$ be the Weil group of $F$. In \cite{dhkm_moduli}, a space $$Z^1(W_F,\widehat{G}) = \Spec(\fR_{^LG})$$ of cocycles was defined over $\ZZ[1/p]$, with a natural action of $\widehat{G}$ by conjugation. A filtration $(P_F^e)_{e\in \mathbb{N}}$ of the wild inertia subgroup $P_F$ expresses $Z^1(W_F,\widehat{G})$ as an ind-scheme over $e$ of $Z^1(W_F/P_F^e, \widehat{G})=\Spec(\fR_{^LG}^e),$ each of which is a finite union of connected components. The base change to $\C$ of the GIT quotient scheme $Z^1(W_F/P_F^e,\widehat{G})\sslash \widehat{G} = \Spec((\fR_{^LG}^e)^{\widehat{G}})$ is an affine variety whose $\C$-points are equivalence classes of semisimple Langlands parameters $\rho^{ss}:W_F\to {^LG}(\C)$ trivial on $P_F^e$.

Let $\ZZ_{\ell}'=\ZZ_{\ell}[\sqrt{q}]$ and consider the base change $\fR_{^LG,\ZZ_{\ell}'}^e$ of $\fR_{^LG}^e$ to $\ZZ_{\ell}'$. Suppose we are given a ring homomorphism
$\mathcal{L}:\fR_{^LG, \ZZ_{\ell}'}^{\widehat{G}}\to \cZ_{G, \ZZ_{\ell}'}.$ A supercuspidal support for $G$ over $\qlb$ is equivalent to a $\qlb$-point of $\cZ_{G,\ZZ_{\ell}'}$, by which we mean a homomorphism $\lambda: \cZ_{G,\ZZ_{\ell}'}\to \qlb$. The pullback $\lambda \circ \mathcal{L}$ defines a $\qlb$-point of $\fR_{^LG,\ZZ_{\ell}'}^{\widehat{G}}$, i.e. a semisimple $\qlb$-valued Langlands parameter for $G$. Therefore $\mathcal{L}$ induces a semisimple local Langlands correspondence over $\qlb$ from the set of supercuspidal supports to the set of semisimple Langlands parameters. In fact, it encodes quite a bit more than just a map on $\qlb$-points, for example, it also transports congruence information modulo ideals. Let $\cZ_{G,\ZZ_{\ell}'}^r$ be the direct factor of $\cZ_{G,\ZZ_{\ell}'}$ corresponding to the subcategory of depth at most $r$. Any morphism $\cL$ must satisfy the following continuity property: there is $e\in \mathbb{N}$ such that the composition
$$\cL:\fR^{\widehat{G}}_{^LG,\ZZ_\ell'}\to \cZ_{G,\ZZ_{\ell}'}\to \cZ_{G,\ZZ_{\ell}'}^r$$ factors through $\left(\fR^e_{^LG,\ZZ_{\ell}'}\right)^{\widehat{G}}$. We call such a homomorphism \emph{continuous} and let $e(r)$ denote the smallest integer $e$ with this property.

Multiple methods for producing such a continuous ring homomorphism have been proposed. For example, Fargues and Scholze have constructed in \cite{fargues_scholze} a Hecke action of a certain derived category of sheaves on the moduli stack $[X_{^LG,\ZZ_{\ell}'}/\widehat{G}]$ on an enlargement of $\Rep_{\ZZ_{\ell}'}(G)$. An element of $(\fR_{^LG,\ZZ_{\ell}'})^{\widehat{G}}$ defines an endomorphism of the structure sheaf of the first category, thus the categorical action associates to it an element of the center of the latter category, i.e. an element of $\cZ_{G,\ZZ_{\ell}'}$. Being functorial, this produces a ring homomorphism $\mathcal{L}_{\text{FS}}$ as above. Another approach is the so-called ``local Langlands in families'' strategy, which takes as an input a classical local Langlands correspondence of $\C$-representations together with some of its expected properties, and tries to interpolate and descend it to a ring homomorphism $\mathcal{L}_{\text{LLiF}}$ over $\ZZ_{\ell}'$. Both methods have been established for $GL_n$ in \cite{helm_curtis, HMconverse} and \cite{fargues_scholze}, respectively, and are proven to be compatible in \cite{fargues_scholze}.

Now let $G$ be a quasisplit classical group. A continuous homomorphism $\cL_G:\fR_{^LG,\ZZ_{\ell}'}^{\widehat{G}}\to \cZ_{G,\ZZ_{\ell}'}$ interpolating a local Langlands correspondence over $\C$ must, if it exists, respect universal $j$-functions as we now explain. First, extend scalars to the Witt ring $W(\Fl)$, fix an additive character $\psi:F\to W(\Fl)^{\times}$, and normalize Haar measures on $U_P$ and $U_{\overline{P}}$ relative to $\psi$ in an appropriate way (\cite[B.2]{gan_ichino_2014}). Let $\K=\text{Frac}(W(\Fl))$ and identify $\C\cong \Kbar$. Let $m$ be the integer defined relative to $G$ in Section~\ref{classicalgroupsdefinitions}, and identify $G\times GL_m$ with a Levi subgroup in a classical group $G'$ of the same type as $G$. For an irreducible smooth $\Kbar[G]$-module $\pi$ and an irreducible smooth $\Kbar[GL_m]$-module $\pi'$, let $j_{\psi}^{G'}(\pi\otimes \pi')$ denote the Harish-Chandra $j$-function with Haar measures normalized relative to $\psi$.  Langlands' conjecture on the Plancherel measure, which is established for classical groups (c.f. \cite[Prop 7.6]{dhkm_conjecture}), implies that 
$$j_{\psi}(\phi_\pi\otimes \phi_{\pi'}) = j_{\psi}^{G'}(\pi\otimes \pi'),$$ where $\phi_\pi$ and $\phi_{\pi'}$ are the Langlands parameters of $\pi$ and $\pi'$ and $j_{\psi}(\phi_\pi\otimes \phi_{\pi'})$ is defined as a certain product of Langlands--Deligne gamma factors (see Subsection~\ref{subsec:conversethm}). The ring $\fR_{^LG,\ZZ_{\ell}'}^e$ supports a universal $\ell$-adically continuous parameter $\phi_{\ell,\univ}^e:W_F/P_F^e\to {^LG}(\fR_{^LG,\ZZ_{\ell}'}^e)$ such that every $\ell$-adically continuous parameter comes from $\phi_{\ell,\univ}^e$ by base change. Given a morphism $\mathcal{L}_G$ over $\ZZ_{\ell}'$ interpolating the classical local Langlands correspondence it must (by checking on $\Kbar$-points and using reducedness) satisfy the analogous equality of \emph{universal} $j$-functions, i.e., after extending scalars to $W(\Fl)$, 
$$(\mathcal{L}_G\otimes \mathcal{L}_{GL_m})(j_{\psi}(\phi_{\univ}^{e(r)}\otimes \phi_{\univ}^{e(r')},X,R,\psi)) = j_{\psi}^{G'}(\cP(G)_{0,r}\otimes \cP(GL_m)_{0,r'}),$$ where $\cP(G)_{0,r}$ denotes the depth $\leq r$ projective generator as above and $j_{\psi}^{G'}(\cP(G)_{0,r}\otimes \cP(GL_m)_{0,r'})$ is the universal Harish-Chandra $j$-function with Haar measures normalized relative to $\psi$.

Our final result is that this compatibility of $j$-functions uniquely characterizes any putative morphism $\cL_G$:
\begin{thm}\label{introthm:characterization}
Fix $\psi:F\to W(\Fl)^{\times}$. There exists at most one continuous homomorphism
$$\mathcal{L}_G:\fR_{^LG, \ZZ_{\ell}'}^{\widehat{G}}\to \cZ_{G, \ZZ_{\ell}'}$$ such that $$(\mathcal{L}_G\otimes \mathcal{L}_{GL_m})(j_{\psi}(\phi_{\univ}^{e(r)}\otimes \phi_{\univ}^{e(r')},X,R,\psi)) = j_{\psi}^{G'}(\cP(G)_{0,r}\otimes \cP(GL_m)_{0,r'}).$$ 
\end{thm}

When $\ell$ is a banal prime, the existence of such an $\cL_G$ follows from \cite[Cor 7.16, Thm 8.2]{dhkm_conjecture}. The compatibility of the Fargues--Scholze correspondence with established cases of the classical local Langlands correspondence is known for classical groups in some specific cases, but is an open problem in general. Theorem~\ref{introthm:characterization} suggests the theory of intertwining operators could be fruitful in approaching this problem.

\section*{Acknowledgements}
The authors are grateful to Johannes Girsch for helpful conversations and to Jean-Fran\c{c}ois Dat, David Helm, Rob Kurinczuk, Thomas Lanard, Nadir Matringe, Alberto M\'inguez, Vincent Sécherre, and Shaun Stevens for their comments and encouragement. The first author was supported by NSF award DMS-2302591 and the University of Maine. The second author benefited from the support of the EPSRC grant EP/V061739/1.

\section{Notation and finiteness properties}\label{sec:notation}

\subsection{Finiteness.} Let $M$ be a reductive group over $F$ and let $A$ be a commutative Noetherian $\mathbb{Z}[1/p]$-algebra. Let $\textup{Rep}_A(M)$ be the category of smooth $A[M]$-modules and denote by $\mathcal{Z}_{A,M}$ the center of $\Rep_A(M)$, i.e., the endomorphism ring $\End(\text{id}_{\Rep_A(M)})$ of the identity functor.

By fixing a Haar measure on $M$ valued in $A$, we can consider the Hecke algebra $\mathcal{H}_A(M) : = C_c^\infty(M,A)$ of locally constant compactly supported $A$-valued functions on $M$ endowed with the convolution product. When $K$ is an open pro-$p$-subgroup of $M$, the relative Hecke algebra $\mathcal{H}_A(M,K)$ is the subalgebra of $\mathcal{H}_A(M)$
of bi-$K$-invariant functions. It has center $\mathcal{Z}_A(M,K)$. One of the equivalent definitions of $\mathcal{Z}_{A,M}$ is
$$\mathcal{Z}_{A,M} = \lim_{\underset{K}{\leftarrow}} \mathcal{Z}_A(M,K)$$ 
where the projective limit runs over open pro-$p$-subgroups in M and the transition maps are surjective, given by $z \in \mathcal{Z}_A(M,K) \mapsto z e_{K'} \in \mathcal{Z}_A(M,K')$ if $K \subseteq K'$ and $e_{K'}$ is the idempotent on $K'$.

\begin{thm}[\cite{dhkm_finiteness,dhkm_conjecture}] \label{finiteness_dhkm_thm} For all open pro-$p$-subgroups $K$ in $M$, we have
\begin{enumerate}
    \item $\mathcal{H}_A(M,K)$ is finitely generated as a module over its center $\mathcal{Z}_A(M,K)$;
    \item $\mathcal{Z}_A(M,K)$ is a finitely generated $A$-algebra.
\end{enumerate}
\end{thm}
In various places, we will need the following finiteness lemma:

\begin{lemma} \label{lem:finiteness} Assume $\sigma \in \textup{Rep}_A(M)$ is finitely generated. Then:
\begin{enumerate}
    \item $\sigma$ is admissible over $\mathcal{Z}_{A,M}$;   
    \item the $\mathcal{Z}_{A,M}$-module $\textup{End}_{A[M]}(\sigma)$ is noetherian.
\end{enumerate}    
\end{lemma}

\begin{proof} (i) Because $\sigma$ is finitely generated, there exists a finite collection of vectors $v_1, \dots, v_n$ which generates $\sigma$. These vectors induce a surjection $\mathcal{H}_A(M)^{\oplus n} \twoheadrightarrow \sigma$ in $\textup{Rep}_A(M)$. We choose an open pro-$p$-subgroup $K$ in $M$ fixing all these vectors. Let $K'$ be an open subgroup $K$. Let $e_K$ and $e_{K'}$ be the idempotent in the Hecke algebras associated to $K$ and $K'$. Then it is easy to see that the previous surjection induces a surjection $\mathcal{H}_A(M,K')^{\oplus n} \twoheadrightarrow \sigma^{K'}$. As $A$ is noetherian, the algebra $\mathcal{H}_A(M,K)$ is finite over its center $\mathcal{Z}_A(M,K)$ by Theorem \ref{finiteness_dhkm_thm}, which is itself a quotient of $\mathcal{Z}_{A,M}$. So $\sigma^{K'}$ is a finitely generated $\mathcal{Z}_{A,M}$-module.

\noindent (ii) Because $\sigma$ is generated by $\sigma^K$, we have $\textup{End}_{A[M]}(\sigma) \hookrightarrow \textup{End}_{\mathcal{H}_A(M,K)}(\sigma^K)$. By Theorem \ref{finiteness_dhkm_thm}, the right-hand side is a finitely generated module over $\mathcal{Z}_A(M,K)$, so it is finitely generated over $\mathcal{Z}_{A,M}$ as earlier. It is even noetherian because the module structure factors through the noetherian ring $\mathcal{Z}_A(M,K)$. Therefore $\textup{End}_{A[M]}(\sigma)$ is noetherian too.
\end{proof}

\begin{rem} 
The ring $\cZ_{A,M}$ is not noetherian, but is a countable product of noetherian rings. Thus, in regard to Lemma~\ref{lem:finiteness}(ii), we point out that a $\cZ_{A,M}$-module is noetherian if and only if it is finitely generated over $\cZ_{A,M}$ and the $\cZ_{A,M}$-structure factors through a finite number of components. 
\end{rem}

\subsection{Flat extension of scalars}\label{subsec:flat_extension} The following proposition is paraphrasing \cite[Prop 2.13]{lam_projective}.

\begin{prop} \label{extension_of_scalars_flat_lam_prop} Let $A \to B$ be a flat morphism and let $V$ and $W$ be in $\textup{Rep}_A(M)$. The natural map $\textup{Hom}_{A[M]}(V,W) \otimes_A B \to \textup{Hom}_{B[M]}(V \otimes_A B , W \otimes_A B)$ of $B$-modules is:
\begin{itemize}
    \item a monomorphism if $V$ is finitely generated;
    \item an isomorphism if $V$ is finitely presented.
\end{itemize} \end{prop}

\begin{proof} To see this is \cite[Prop 2.13]{lam_projective}, note that there exists $r \geq 0$ such that $V$ has depth at most $r$ if $V$ is finitely generated. The category $\textup{Rep}_A^{\leq r}(M)$ of representation of depth at most $r$ is equivalent to a category of modules over a direct factor ring $e_{\leq r} \mathcal{H}_A(M,K)$ of $\mathcal{H}_A(M,K)$ where $K$ is an open pro-$p$-subgroup that is small enough. So the proposition of Lam applies. \end{proof}

When $A$ is noetherian, the finiteness of Hecke algebras from Theorem \ref{finiteness_dhkm_thm} ensures that finitely generated representations are also finitely presented. In what follows, we will repeatedly use the compatibility property in Proposition~\ref{extension_of_scalars_flat_lam_prop} without comment.

By Lemma~\ref{lem:finiteness}, we can apply the following property to the centers of endomorphism rings of finitely generated objects:
\begin{lemma}[{\cite[Lem A.2]{dhkm_conjecture}}]
Let $\mathcal{E}$ be an $A$-algebra with center $Z(\mathcal{E})$ such that $\mathcal{E}$ is finitely generated over $Z(\mathcal{E})$. Let $A'$ be a flat commutative $A$-algebra. Then: $$Z(\mathcal{E})\otimes_AA'=Z(\mathcal{E}\otimes_AA').$$
\end{lemma}

\subsection{Free extension of scalars} Proposition \ref{extension_of_scalars_flat_lam_prop} can be improved when the extension is free.

\begin{prop} \label{prop:free-extension-scalars-hom-fg} Let $B$ be an $A$-algebra such that $B$ is free as an $A$-module. Let $V$ and $W$ be in $\textup{Rep}_A(M)$. The natural map  $\textup{Hom}_{A[M]}(V,W) \otimes_A B \to \textup{Hom}_{B[M]}(V \otimes_A B , W \otimes_A B)$ of $B$-modules is bijective if $V$ is finitely generated or if $B$ is finite free over $A$. \end{prop}

\begin{proof} Let $\phi \in \textup{Hom}_{B[M]}(V \otimes_A B , W \otimes_A B)$. Then $V \otimes_A B \simeq \oplus_\mathcal{B} V$ as $A[M]$-modules by choosing a basis $\mathcal{B}$ of $B$ over $A$. Because $\phi$ is $B$-linear, its restriction to any copy of $V$ surely determines $\phi$. We fix such a copy of $V$ from now on. Conversely, any morphism $\textup{Hom}_{A[M]}(V,W \otimes_A B)$ can be extended in a unique way as a morphism of $V \otimes_A B$ that is $B$-linear. In other words, we have a $B$-module isomorphism
$$\textup{Hom}_{B[M]}(V \otimes_A B,W \otimes_A B) \simeq \textup{Hom}_{A[M]}(V,W \otimes_A B).$$
As $W \otimes_A B \simeq \oplus_\mathcal{B} W \subseteq \prod_\mathcal{B} W$, we obtain an embedding 
$$\textup{Hom}_{B[M]}(V\otimes_A B,W\otimes_A B) \hookrightarrow \textup{Hom}_{A[M]}(V,\prod_\mathcal{B} W) = \prod_\mathcal{B} \textup{Hom}_{A[M]}(V,W).$$
We write $(\phi_\lambda)$ for the image of $\phi$ where $\lambda$ runs over $\mathcal{B}$.

The image of $\phi \in \textup{Hom}_{A[M]}(V,W) \otimes_A B$ into this product corresponds to elements of the form $(\phi_\lambda)$ where all but finitely many $\phi_\lambda$'s are non zero. The image of $\phi \in \textup{Hom}_{B[M]}(V \otimes_A B,W \otimes_A B)$ consists of $(\phi_\lambda)$ where $I_{\phi,v} = \{ \lambda \ | \ \phi_\lambda(v) \neq 0 \}$ is a finite set for all $v \in V$ (or equivalently $\phi(v)$ belongs the direct sum). It is clear with this description that 
$$\textup{Hom}_{A[M]}(V,W) \otimes_A B \hookrightarrow \textup{Hom}_{B[M]}(V \otimes_A B,W \otimes_A B).$$
Of course, if the direct sum and the product over $\mathcal{B}$ agree \textit{i.e.} $B/A$ is finite, then all maps above are isomorphisms and the proposition follows.

For $\phi \in \textup{Hom}_{B[M]}(V \otimes_A B,W \otimes_A B)$, the support $I_\phi = \cup_{v \in V} I_{\phi,v}$ is the set of $\lambda$'s such that $\phi_\lambda \neq 0$. When $V$ is finitely generated, we take $v_1$, \dots, $v_n$ generating $V$ and notice that $I_\phi = \cup_i I_{v_i}$. Therefore $I_\phi$ is finite \textit{i.e.} all but finitely many $\phi_\lambda$'s are non zero. The result follows. \end{proof}

\subsection{Geometric lemma}\label{subsec:geometric} In this subsection, let $R$ be a commutative algebra over $\ZZ' = \ZZ[\sqrt{q}^{-1}]$. Fix a maximal split torus $A_0$ of $G$. Its centralizer $M_0$ is a Levi factor of a minimal parabolic $P_0$. We say a parabolic subgroup $P$ is standard if $P\supset A_0$ and semistandard if $P\supset P_0$; in either case, $P$ has a unique Levi containing $A_0$, such a Levi is called standard (resp. semistandard) if $P$ is standard (resp. semistandard). In this setup we can define $W_M$ and $W_G$ in the usual way and have a canonical injection $W_M\hookrightarrow W_G$.

The assumption that $M$ be a standard Levi subgroup is convenient for formulating the geometric lemma. It does not cause any loss of generality in our final results because every Levi subgroup is $G$-conjugate to a standard Levi subgroup (c.f. Remark~\ref{rem:drop_standard}). Therefore, outside of this subsection (including the introduction), we frequently omit the word ``standard'' with the understanding that the reader can adapt the proofs to the general case, or reduce to the standard case, as needed.

Let $P$ and $Q$ be two semistandard parabolic subgroups with a common standard Levi subgroup $M$ and let $\sigma$ be a smooth $R[M]$-module. The geometric lemma gives a filtration of functors on the category of smooth $R[M]$-modules,
$$0= F_0\subset F_1\subset \cdots \subset F_k = r_Q^Gi_P^G,$$
and isomorphisms $$F_i/F_{i-1} \cong i^M_{M\cap w_i^{-1}P}\circ \dot{w_i} \circ r^M_{M\cap w_iQ},$$ where $\dot{w_i}$ are representatives of the double cosets $w_i\in P\backslash G/Q$. These double cosets are in bijection with the set $W_M\backslash W_G/W_M$, which does not depend on $P$ or $Q$, but we will equip $W_M\backslash W_G/W_M$ with a linear ordering that refines the Bruhat partial ordering relative to $P$ and $Q$, which is defined as
$$w'\leq w\text{ if } Pw' Q\subset \overline{ PwQ},$$ the closure being taken in the $p$-adic topology on $G$. Choose a linear ordering, denoted by $\preceq$, that refines the Bruhat partial ordering. If $Q=P$, then $1$ represents the smallest element of $W_M\backslash W_G/W_M$; if $Q = \overline{P}$, it represents the largest.

Define $$X_{\prec w} = \bigcup_{w'\prec w} Pw'Q \ \ \ ,\ \ \ \ X_{\preceq w} = \bigcup_{w'\preceq w} Pw'Q\ ,$$ and the following subfunctors of $i_P^G$:
\begin{align*}
\widetilde{F}_{PQ}^{\prec w}(\sigma)&:= \{f\in i_P^G(\sigma)\ :\ \supp(f)\cap X_{\prec w}= \emptyset\}\\
\widetilde{F}_{QP}^{\preceq w}(\sigma) &:= \{f\in i_P^G(\sigma)\ :\ \supp(f)\cap X_{\preceq w} = \emptyset\}\\
\widetilde{F}_{QP}^{< w}(\sigma) &:= \{f\in i_P^G(\sigma)\ :\ \supp(f)\cap \overline{PwQ}\subset PwQ\}\\
\widetilde{F}_{QP}^{\leq w}(\sigma) &:= \{f\in i_P^G(\sigma)\ :\ \supp(f)\cap \overline{PwQ}=\emptyset\}
\end{align*}
We have $\widetilde{F}_{QP}^{\preceq w} \subset \widetilde{F}_{QP}^{\prec w}\subset \widetilde{F}_{QP}^{<w}$ and $w_1\prec w_2 \implies \widetilde{F}_{QP}^{\prec w_2}\subset \widetilde{F}_{QP}^{\prec w_1}$.

If $\widetilde{F}_{QP}^{(-)}$ is one of the four functors defined above, we write $F_{QP}^{(-)}$ for its image in the parabolic restriction $r_Q^Gi_P^G$. Thus the filtration $F_0\subset \cdots\subset F_k$ of the geometric lemma is achieved by enumerating $W_M\backslash W_G/W_M$ in reverse size order $w_0 \succ \cdots \succ w_k$ and there are isomorphisms
$$I_w := F_{QP}^{< w}/F_{QP}^{\leq w} \cong F_{QP}^{\prec w}/F_{QP}^{\preceq w} \cong i^M_{M\cap w^{-1}P}\circ\dot{w} \circ r^M_{wQ\cap M}.$$ We will not need to write down these isomorphisms explicitly except when $w = 1$ where the isomorphism $I_1 \overset{\sim}{\to} \text{id}$ is induced by the map $f\mapsto \int_{U_Q\cap U_{\overline{P}}}f(u)du$ for $f$ in $\widetilde{F}^{<1}_{QP}$.

When $Q=P$, observe that $F_{QP}^{< 1}$ is $r_Q^Gi_P^G$ and $F_{QP}^{\leq 1}$ is the image under parabolic restriction of the functions supported away from $P$. 

When $Q = \overline{P}$ is the opposite parabolic, $F_{QP}^{<1}$ is the image under parabolic restriction of the functions supported on $P\overline{P}$ and $F_{QP}^{\leq 1}=0$.

\section{Unramified twisting and intertwining operators}\label{sec:unramifiedtwisting}

\subsection{The universal unramified twist}
Let $A$ be a Noetherian commutative algebra over $\ZZ' = \ZZ[\sqrt{q}^{-1}]$. Let $P$ and $Q$ be two parabolic subgroups with a common Levi subgroup $M$, and let $\sigma_0$ be a smooth $A[M]$-module. An intertwining operator, broadly speaking, is any element of $\Hom_{A[G]}(i_P^G(\sigma_0), i_Q^G(\sigma_0))$, however intertwining operators are studied and used most fruitfully as $\sigma_0$ varies within a family of representations defined by unramified twisting. More precisely, let
$$M^0 = \bigcap_{\chi}\ker|\chi|_F$$ as $\chi$ varies over the $F$-rational characters $\chi:M\to F^{\times}$, and for $m\in M$ let $\overline{m}$ denote its image in the free finite-rank abelian group $M/M^0$. The ring
$$R:= A[M/M^0]$$ is a commutative $A$-algebra, isomorphic to a polynomial ring in $r$ variables, where $r$ is the rank of $M/M^0$. The universal unramified character of $M$ is defined as
\begin{align*}
\Psi_M:M&\to R^{\times}\\
m&\mapsto \overline{m} = \Psi_M(m)\ .
\end{align*}
We will denote $\Psi_M$ simply by $\Psi$ if $M$ has been fixed. The universal unramified twist of $\sigma_0$ will be the $R[M]$-module defined by
$\sigma := \sigma_0\otimes_A \Psi.$ The underlying $R$-module is $\sigma_0 \otimes_A R$ and the action of $M$ is given by $$\sigma(m)(v\otimes 1) = \sigma_0(m)v\otimes \overline{m},\ \ \ v\in \sigma_0,\ m\in M.$$ The ``universal'' terminology is justified by the fact that for any $A$-algebra $\kappa$ that is a field, unramified characters $\chi:M\to \kappa^{\times}$ are equivalent to ring homomorphisms $R\to \kappa$ by way of the specialization $$\chi = \Psi\otimes_{R,\lambda}\kappa.$$ We will be particularly concerned with intertwining operators in $\Hom_{R[G]}(i_P^G(\sigma),i_Q^G(\sigma))$.

The Bernstein center $\cZ_{A,M}$ over $A$ is defined as the center of the category $\Rep_A(M)$, that is,
$$\cZ_{A,M} := \End(\text{id}_{\Rep_A(M)}).$$ It acts on every object of $\Rep_A(M)$ in a natural way. We will consider $\sigma_0$ as a $\cZ_{A,M}[M]$-module. 

Define the commutative $R$-algebra $$\fR = \cZ_{A,M}[M/M^0].$$ We will consider $\sigma$ as an $\fR[M]$-module where $M$ acts as above and $\fR$ acts as 
\begin{align*}
(z \overline{m})(v\otimes 1) &= zv\otimes \overline{m},\ v\in \sigma_0,\ z\in \cZ_{A,M},\ m\in M
\end{align*}

\begin{rem}\label{rem:differentaction}
Note that $\sigma$, by virtue of being an $R[M]$-module, carries an action of the center $\cZ_{R,M}$ of the category $\Rep_R(M)$, and there is a natural ring homomorphism $\fR\to \cZ_{R,M}$. We emphasize that this $\fR$ action is different from the one defined above
\end{rem}

\subsection{$\fR$-linear intertwining operators}
Our first results concern the subset $\Hom_{\fR[G]}(i_P^G(\sigma),i_Q^G(\sigma))$ of $\Hom_{R[G]}(i_P^G(\sigma),i_Q^G(\sigma))$. Composing $f\in i_Q^G(\sigma)$ with $f\mapsto f(1)$ induces the Frobenius reciprocity isomorphism,
$$\Hom_{\fR[G]}(i_P^G(\sigma),i_Q^G(\sigma)) \cong \Hom_{\fR[M]}(r_Q^Gi_P^G(\sigma),\sigma),$$ and we will frame our results in terms of $\Hom_{\fR[M]}(r_Q^Gi_P^G(\sigma),\sigma)$. We will freely use the notation introduced in Subsection~\ref{subsec:geometric}.

\begin{thm}\label{thm:intertwining}
Suppose $\sigma_0$ is finitely generated over $A[M]$. For any $w\neq 1$, $\Hom_{\fR[M]}(I_w(\sigma),\sigma)=0$.
\end{thm}
The proof of Theorem~\ref{thm:intertwining} appears in Section~\ref{sec:proofs}, below. For now, let us take Theorem~\ref{thm:intertwining} for granted and use it to make the following definition.

\begin{definition}\label{def:signature}
Given an intertwining operator $J$ in $\Hom_{\fR[M]}(r_Q^Gi_P^G(\sigma),\sigma)$, Theorem~\ref{thm:intertwining} implies $J$ is zero on $F_{QP}^{\leq 1}(\sigma)$. Thus restricting $J$ to $F_{QP}^{<1}(\sigma)$ defines an endomorphism $\sigma \cong I_1(\sigma) \to \sigma$, which we call the $(\sigma,P,Q)$-signature of $J$ and denote $\sig_{PQ}(J)$.
\end{definition}

\begin{corollary}\label{cor:sig_injective}
An intertwining operator is uniquely determined by its $(\sigma,P,Q)$-signature, i.e., the map of $\fR$-modules $$\sig_{PQ}:\Hom_{\fR[M]}(r_Q^Gi_P^G(\sigma),\sigma)\to \End_{\fR[M]}(\sigma)$$ is injective.
\end{corollary}

Observe that there is a natural left action of $\End_{\fR[M]}(\sigma)$ on $\Hom_{\fR[M]}(r_Q^Gi_P^G(\sigma),\sigma)$ via its action on $\sigma$. Given any intertwining operator $J$ in $\Hom_{\fR[M]}(r_Q^Gi_P^G(\sigma),\sigma)$ the $(\sigma,P,Q)$-signatures of its orbit form a principle left-ideal in $\End_{\fR[M]}(\sigma)$, namely
$$\sig_{PQ}(\End_{\fR[M]}(\sigma)\cdot J) = \End_{\fR[M]}(\sigma)\cdot \sig_{PQ}(J)\subset \End_{\fR[M]}(\sigma).$$

Let us apply this when $Q=P$. In this case the isomorphism $I_1\overset{\sim}{\to} \text{id}$ in \ref{subsec:geometric} is induced by $f\mapsto f(1)$ on $\widetilde{F}_{PP}^{<1}$. Since $F_{PP}^{<1}=r_P^Gi_P^G$, this defines a quotient map
\begin{align*}
J_{P|P}(\sigma_0):r_P^Gi_P^G\sigma&\to \sigma\\
f&\mapsto f(1),\ 
\end{align*}
and we have $\sig_{PP}(J_{P|P}(\sigma_0))=\text{id}_{\sigma}$ by definition. Since $$\sig_{PP}(\End_{\fR[M]}(\sigma)\cdot J_{P|P}(\sigma_0)) = \End_{\fR[M]}(\sigma)\cdot \sig_{PP}(J_{P|P}(\sigma_0)) = \End_{\fR[M]}(\sigma),$$ the map $\sig_{PP}$ is surjective. Now, under Frobenius reciprocity, $J_{P|P}(\sigma_0)$ corresponds to the identity endomorphism $i_P^G(\sigma)\to i_P^G(\sigma)$, so we deduce
\begin{corollary}\label{cor:schurslemmafrakR}
The composition of $\sig_{PP}$ with Frobenius reciprocity $$\End_{\fR[G]}(i_P^G(\sigma))\cong \Hom_{\fR[M]}(r_P^Gi_P^G(\sigma),\sigma)\xrightarrow{\sig_{PP}} \End_{\fR[M]}(\sigma)$$ defines an inverse to the natural map $\End_{\fR[M]}(\sigma)\to \End_{\fR[G]}(i_P^G(\sigma))$ given by the induction functor $i_P^G$.
\end{corollary}
In fact, when $P\neq Q$ we can still recover such an isomorphism after performing a localization.

\subsection{Localizing and preserving congruences}\label{sec:localizing}

Since $\Hom_{\fR[M]}(r_Q^Gi_P^G(\sigma),\sigma)$ is an $\fR$-module, it gives rise to the sheaf of $\mathcal{O}_{\Spec(\fR)}$-modules $$\mathfrak{Int} : U \mapsto \Hom_{\fR[M]}(r_Q^Gi_P^G(\sigma),\sigma) \otimes_{\mathfrak{R}} \mathcal{O}_{\textup{Spec}(\mathfrak{R})}(U).$$
When $b \in \mathfrak{R}$, the sections over the principal open set $U_b = \textup{Spec}(\fR[1/b])$ are given, thanks to Proposition \ref{extension_of_scalars_flat_lam_prop}, by $\mathfrak{Int}(U_b) = \Hom_{\fR[1/b][M]}(r_Q^Gi_P^G(\sigma[1/b]),\sigma[1/b])$. We want to describe $\mathfrak{Int}$ on as large a subset of $\Spec(\fR)$ as possible. For this, we should look for a $J$ satisfying the following conditions:
\begin{enumerate}[(i)]
\item $\sig_{PQ}(J)$ acts on $\sigma$ by a ``scalar'' $b$ in $\fR$ (in particular it is in the center of $\End_{\fR[M]}(\sigma)$),
\item \label{b_preserve_congruences_cond_ii} $b$ is not a zero-divisor in $\fR$ (in particular, $b$ is nonzero). 
\end{enumerate}
Indeed, in this situation, the map $\sig_{PQ}$ would induce an isomorphism
$$\Hom_{\fR[1/b][M]}(r_Q^Gi_P^G(\sigma[1/b]),\sigma[1/b])\to \End_{\fR[1/b][M]}(\sigma[1/b]).$$ In geometric terms, $b$ not being a zero divisor implies (is equivalent to, if $\fR$ is reduced) that the principal open set $\Spec(\fR[1/b])$ is open and dense in $\Spec(\fR)$, so we would have described the sheaf $\Hom_{\fR[M]}(r_Q^Gi_P^G(\sigma),\sigma)$ on a large open subset of $\Spec(\fR)$.

In fact, we should ask for more than \ref{b_preserve_congruences_cond_ii}. The geometry of the Bernstein scheme $\Spec(\cZ_{A,M})$ encodes much of the structure of $\Rep_A(M)$, and we should ask that our families of intertwining operators vary across both $\Spec(\cZ_{A,M})$ and $\Spec(A[M/M^0])$ in a compatible way. Further, allowing $A$ to be a $\ZZ'$-algebra (instead of just a $\C$-algebra) enables us to consider congruences. In this setup, if $b$ were divisible by some prime $\ell\neq p$, passing to $\fR[1/b]$ would kill all congruence information at the mod $\ell$ fiber. Thus it would be particularly valuable to find a $J$ that preserves congruences in the sense of the following stronger condition
\begin{enumerate}
\item[(ii$'$)] For all $\mathfrak{p}\in \Spec(A)$, the image of $b$ in $\fR\otimes_AA/\mathfrak{p}$ is not a zero divisor (hence nonzero).
\end{enumerate}
In geometric terms, the open dense set $U_b = \textup{Spec}(\fR[1/b]) \to \textup{Spec}(\mathfrak{R})$ behaves nicely under pullback over $A$ in the sense that the pullback of $U_b \to \textup{Spec}(\mathfrak{R}) \to \textup{Spec}(A)$ along $\textup{Spec}(A/\mathfrak{p}) \to \textup{Spec}(A)$ gives an open dense set again $\textup{Spec}((\fR\otimes_AA/\mathfrak{p}) [1/b]) \to \textup{Spec}(\mathfrak{R}\otimes_AA/\mathfrak{p})$.

\begin{definition}
Let $\fS$ be the multiplicative subset of $\fR=\cZ_{A,M}[M/M^0]$ consisting of elements of the form $$c_d\overline{m}^d + c_{d-1}\overline{m}^{d-1} + \cdots + c_1\overline{m}+c_0,\ \ m\in M\cap G^0,\ c_i\in \cZ_{A,M}$$ such that both $c_d$ and $c_0$ are units in $\cZ_{A,M}$.
\end{definition}
\begin{obs}
Elements of $\fS$ satisfy condition \textup{(ii$'$)}. 
\end{obs}

Let $\fR[A_M]$ be the group ring of $A_M$, where $A_M$ is the split component of the center of $M$. The ideal
$$
\fI_{\sigma}^{QP} := \ker\bigg(\fR[A_M]\to \End_{\fR[M]}((r_Q^Gi_P^G/F_{QP}^{<1})(\sigma))\bigg)
$$
of $\fR[A_M]$ was considered in \cite[Sec 2.9]{dat} and \cite{dhkm_conjecture}[App B] (take ``$R$'' in \cite{dat, dhkm_conjecture} to be our $\fR$) and provides a source of intertwining operators. Indeed, any element $\beta$ of $\fI_\sigma^{QP}$ induces a homomorphism $$r_Q^Gi_P^G(\sigma) \to F_{QP}^{<1}(\sigma) \to I_1(\sigma)$$ and by composition a map
\begin{align*}
r_Q^Gi_P^G(\sigma)\to I_1(\sigma)&\to \sigma\\
\overline{f}  \mapsto \beta\overline{f} &\mapsto \int_{U_Q\cap U_{\overline{P}}}\beta\overline{f}(u)du\ ,
\end{align*}
where $\overline{f}$ denotes the image of an element $f$ under $i_P^G(\sigma)\to r_Q^Gi_P^G(\sigma)$. While $\beta \overline{f}$ is not, strictly speaking, a function on $G$, we can lift it to $\widetilde{F}_{QP}^{<1}(\sigma)$ and then take the integral, which factors through $I_1(\sigma)$ and does not depend on the choice of lift. Thus we get an element of $\Hom_{\fR[M]}(r_Q^Gi_P^G(\sigma),\sigma)$, which we will call $J_{\beta}$. 
\begin{thm}\label{thm:image} 
There exists $\beta \in \fI_\sigma^{QP}$ such that $\sig_{PQ}(J_\beta)$ acts on $\sigma$ by a scalar $b$ in the multiplicative set $\fS$.
\end{thm}
The proof of Theorem~\ref{thm:image} appears in Section~\ref{sec:proofs} below. For now, we take it for granted and use it to define intertwining operators.

\begin{corollary}\label{cor:isomorphism} Let $b\in \fS$ be as in Theorem~\ref{thm:image}. After localization, the $(\sigma,P,Q)$-signature defines an isomorphism
$$\sig_{PQ}:\Hom_{\fR[1/b][M]}(r_Q^Gi_P^G(\sigma[1/b]),\sigma[1/b])\overset{\sim}{\to} \End_{\fR[1/b][M]}(\sigma[1/b]).$$

\end{corollary}
The isomorphism of Corollary~\ref{cor:isomorphism} will allow us to define the canonical intertwining operator, however, to make such a definition we would like a framework that is independent of the choice of $\beta$ or of $b$. We accomplish this by passing from $\fR[1/b]$ to the fraction ring $\fS^{-1}\fR$, to get an isomorphism:
$$\sig_{PQ}:\Hom_{\fS^{-1}\fR[M]}(r_Q^Gi_P^G(\fS^{-1}\sigma),\fS^{-1}\sigma)\overset{\sim}{\to} \End_{\fS^{-1}\fR[M]}(\fS^{-1}\sigma).$$
\begin{definition}
The canonical intertwining operator, denoted $J_{Q|P}(\sigma_0)$, is the unique element of $\Hom_{\fS^{-1}\fR[G]}(i_P^G(\fS^{-1}\sigma),i_Q^G(\fS^{-1}\sigma))$ whose image in
$$\Hom_{\fS^{-1}\fR[M]}(r_Q^Gi_P^G(\fS^{-1}\sigma),\fS^{-1}\sigma)\overset{\sig_{PQ}}{\to}\End_{\fS^{-1}\fR[M]}(\fS^{-1}\sigma)$$ is the identity.
\end{definition}

\begin{rem}\label{rem:terminology}
Let $\beta$, $J_{\beta}$ and $b$ be as in Theorem~\ref{thm:image}. Let $\fI_{\sigma}$ be the ideal of $\fR[A_M]$ that acts by $0$ on $\sigma$. Since $$b = (b-\beta) + \beta \in \fI_{\sigma} + \fI_{\sigma}^{QP},$$ and $b$ is a unit in $\fS^{-1}\fR$, one would say in the terminology of \cite{dat}[Sec 2] that $\fS^{-1}\sigma$ is $(P,Q)$-regular as a $\fS^{-1}\fR[M]$-module. Or, in the terminology of \cite[Appendix B]{dhkm_conjecture}, one could say that $b$ is $(\sigma,P,Q)$-singular when $\sigma$ is considered as an $\fR[M]$-module (c.f. \cite[Lem B.3]{dhkm_conjecture}). 
\end{rem}
\subsection{Classical interpretation}\label{subsec:classicalinterpretation}
We can unwind the constructions above to characterize $J_{Q|P}(\sigma_0)$ in terms of the familiar definition of intertwining operators via convergent integrals. Let $\beta$, $J_{\beta}$ and $b$ be as in Theorem~\ref{thm:image}. Restricting $\beta$ to the submodule $F_{QP}^{<1}(\sigma)\subset r_Q^Gi_P^G(\sigma)$ gives a homomorphism
$$F_{QP}^{<1}(\sigma)\to I_1(\sigma),$$ which by Theorem~\ref{thm:intertwining} is zero on $F_{QP}^{\leq 1}(\sigma)$ and defines an element of $\End_{\fR[M]}(I_1(\sigma))\cong \End_{\fR[M]}(\sigma)$. The action of $A_M$ on $\sigma$ is via the map
\begin{align*}
A_M&\to \fR[M]\\
a &\mapsto z_a\cdot a
\end{align*} where $z_a$ is the element of $\cZ_{A,M}$ defined by $a$. The element $\beta\in \fR[A_M]$ acts on $\sigma$ by a ``scalar'' $b$ in the multiplicative set $\fS\subset\fR$. The element $J_{\beta}\in \Hom_{\fR[M]}(r_Q^Gi_P^G(\sigma),\sigma)$ is defined by
$$\overline{f}\mapsto \int_{U_Q\cap U_{\overline{P}}}(\beta \overline{f})(u)du\ ,\ \text{$\overline{f}\in r_Q^Gi_P^G(\sigma),$}$$ and $\frac{1}{b}\sig_{PQ}(J_{\beta})$ is the identity in $\End_{\fS^{-1}\fR[M]}(\fS^{-1}\sigma)$.

Translating through Frobenius reciprocity, we get
$$J_{Q|P}(\sigma_0)(f)(g) = \frac{1}{b}\int_{U_Q\cap U_{\overline{P}}}\beta(\overline{gf})(u)du,\ \ f\in i_P^G(\mathfrak{S}^{-1} \sigma).$$ If we start with $f\in \widetilde{F}_{QP}^{<1}(\sigma)$ then, because $\beta$ and $b$ coincide on $I_1(\sigma)$, we can cancel and simplify the expression to $$J_{Q|P}(\sigma_0)(f)(g) = \int_{U_Q\cap U_{\overline{P}}}f(ug)du.$$

\begin{corollary}\label{cor:characterization}
The operator $J_{Q|P}(\sigma_0)$ is the unique $\fS^{-1}\fR[G]$-equivariant homomorphism $$i_P^G(\fS^{-1}\sigma)\to i_Q^G(\fS^{-1}\sigma)$$ satisfying $J_{Q|P}(\sigma_0)(f)(1) = \int_{U_Q\cap U_{\overline{P}}}f(u)du$ for $f\in \widetilde{F}_{QP}^{<1}(\fS^{-1}\sigma)$.
\end{corollary}
\begin{proof}
By Remark~\ref{rem:terminology} above, we have shown that $\fS^{-1}\sigma$ is $(P,Q)$-regular as a $\fS^{-1}\fR[M]$-module, in the terminology of \cite{dat}. The corollary then follows from the formalism of loc. cit., specifically \cite[Lem 7.12]{dat}.
\end{proof}

\subsection{Compatibility of intertwining operators with subquotients and scalar extension
}\label{sec:compatibility_of_intertwiningoperators}
Given $A\to A'$ a homomorphism of commutative rings, the forgetful functor induces a natural map $\cZ_{A,M}\to \cZ_{A',M}$, which factors through $\cZ_{A,M}\to \cZ_{A,M}\otimes_AA'$. Let $f:R\to R'$ (resp. $\mathfrak{f}:\fR\to \fR'$) be the corresponding induced homomorphisms $A[M/M^0]\to A'[M/M^0]$ (resp. $\cZ_{A,M}[M/M^0]\to \cZ_{A',M}[M/M^0]$). The image of the multiplicative set $\fS$ under $\mathfrak{f}$ is contained in the multiplicative set $\fS'$ of $\fR'$ with the analogous definition.

\begin{prop}\label{prop:compatibility}
\begin{enumerate}
\item Let $q_0:\sigma_0\to \sigma_0'$ be a homomorphism of finitely generated $A[M]$-modules, let $q:\fS^{-1}\sigma\to \fS^{-1}\sigma'$ be the induced homomorphism of $\fS^{-1} \mathfrak{R}[M]$-modules. The following diagram commutes.
$$
\begin{tikzcd}
i_P^G (\fS^{-1}\sigma) \arrow[r,"J_{Q|P}(\sigma_0)"] \arrow[d,"i_P^G(q)", rightarrow]& i_{Q}^{G}(\fS^{-1}\sigma) \arrow[d,"i_P^G(q)",rightarrow]\\
i_P^G (\fS^{-1}\sigma') \arrow[r,"J_{Q|P}(\sigma_0')"] & i_{Q}^{G}(\fS^{-1}\sigma')
\end{tikzcd}
$$
\item Let $\sigma_{0,A'} = \sigma_0\otimes_AA'$ and $\sigma_{R'}=\sigma\otimes_RR'$ and let $\iota:\fS^{-1}\sigma\to {\fS'}^{-1}\sigma_{R'}$ denote the map on localizations induced by $\textup{id}\otimes 1:\sigma\to \sigma_{R'}$. The following diagram commutes
$$
\begin{tikzcd}
i_P^G (\fS^{-1}\sigma) \arrow[r,"J_{Q|P}(\sigma_0)"] \arrow[d,"i_P^G(\iota)"]& i_{Q}^{G}(\fS^{-1}\sigma) \arrow[d,"i_Q^G(\iota)"]\\
i_P^G ({\fS'}^{-1}\sigma_{R'}) \arrow[r,"J_{Q|P}(\sigma_{0,A'})"] & i_{Q}^{G}({\fS'}^{-1}\sigma_{R'})
\end{tikzcd}$$
\end{enumerate}
\end{prop}
\begin{proof}
The integral appearing in Corollary~\ref{cor:characterization} is a finite sum when restricted to $\widetilde{F}_{QP}^{<1}(\fS^{-1}\sigma)$ and commutes with homomorphisms $\sigma\to\sigma'$ and with scalar extension.  
\end{proof}

Let $\Rep_A^{\text{f.g.}}(M)$ denote the category of finitely generated smooth $A[M]$-modules. Consider the following functor
\begin{align*}
\fri_P^G: \Rep_A^{\text{f.g.}}(M)&\to \Rep_{\fS^{-1}\fR}(G)\\
\sigma_0 &\to i_P^G(\fS^{-1}\sigma)\ .
\end{align*}
Since taking universal unramified twist, parabolic induction, and localization are all exact functors, we observe that $\fri_P^G$ is exact. In this language, Proposition~\ref{prop:compatibility}(i) says that $J_{Q|P}$ defines a natural transformation between exact functors:
$$J_{Q|P}:\fri_P^G\to \fri_Q^G.$$

The same proof as Proposition~\ref{prop:compatibility} also implies the intertwining operators satisfy compatibility with specialization at points of the Bernstein center. More precisely, consider a homomorphism of commutative Noetherian $A$-algebras $\lambda:\cZ_{A,M}\to B$. It induces a homomorphism $\widetilde{\lambda}:\fR \to \cB$, where $\cB:= B[M/M^0]$, and $\widetilde{\lambda}$ takes the multiplicative set $\fS$ to the multiplicative subset of $\cS$ of $\cB$ having the analogous definition.
\begin{prop}\label{prop:bernsteinscalarextension}
Let $\sigma_\lambda=\sigma\otimes_{\fR,\widetilde{\lambda}}\cB$ and let $\iota:\fS^{-1}\sigma\to {\cS}^{-1}\sigma_{\lambda}$ denote the map on localizations induced by $\textup{id}\otimes 1:\sigma\to \sigma_{\lambda}$. The following diagram commutes
$$
\begin{tikzcd}
i_P^G (\fS^{-1}\sigma) \arrow[r,"J_{Q|P}(\sigma_0)"] \arrow[d,"i_P^G(\iota)"]& i_{Q}^{G}(\fS^{-1}\sigma) \arrow[d,"i_Q^G(\iota)"]\\
i_P^G ({\cS}^{-1}\sigma_\lambda) \arrow[r,"J_{Q|P}(\sigma_{0,\lambda})"] & i_{Q}^{G}({\cS}^{-1}\sigma_{\lambda})
\end{tikzcd}.$$
\end{prop}

\begin{rem}
The level of generality in Proposition~\ref{prop:compatibility} is an improvement on \cite[Lem B.8]{dhkm_conjecture} and \cite[Lem 7.2]{dat} .
\end{rem}

\subsection{Factorization property of intertwining operators and compatibility with induction}

The next two lemmas show the canonical intertwining operator $J_{Q|P}$ is compatible with respect to changing $Q$, in a certain sense, and with parabolic induction. The arguments in \cite[Prop 7.8]{dat} work in this setting. See also \cite[IV.1(12),(14)]{waldspurger} for the classical formulation.

Given semistandard parabolics $P$, $Q$, we define $d(P,Q) = |\Sigma_{red}(P)\cap \Sigma_{red}(\overline{Q})|$ where $\Sigma_{red}(P)$ denotes the set of reduced roots of $A_M$ in $P$.

\begin{lemma}\label{lem:multiplicativity_of_intertwiners}
Let $O$, $P$, $Q$ be three parabolics with Levi component $M$ such that $d(O,Q) = d(O,P)+d(P,Q)$ and $d(O,P)=1$. Then
$$J_{Q|O}(\sigma_0) = J_{Q|P}(\sigma_0)\circ J_{P|O}(\sigma_0).$$
\end{lemma}

\begin{lemma}[Compatibility with induction]\label{lem:factorization_of_intertwiners}
Let $P$, $Q$ be two parabolics with Levi component $M$.
\begin{enumerate}
\item Suppose $P$ and $Q$ are contained in a parabolic subgroup $O$ with Levi $N$. Then
$$J_{Q|P}^G(\sigma_0) = i_O^G(J_{Q\cap N|P\cap N}^N(\sigma_0)).$$
\item Let $N$ be a Levi subgroup of $G$ containing $M$ such that $P\cap N = Q\cap N$ and such that $PN$ and $QN$ are parabolic subgroups with Levi component $N$. Then
$$J_{Q|P}(\sigma_0) = J_{QN|PN}(i_{P\cap N}^N(\sigma_0)).$$
\end{enumerate}
\end{lemma}

\section{Proofs of Theorems~\ref{thm:intertwining} and~\ref{thm:image}}\label{sec:proofs}

\begin{proof}[Proof of Theorem~\ref{thm:intertwining}]
We show that $\textup{Hom}_{B[M]}(I_w(\sigma),\sigma) = 0$ where $I_w(\sigma) =  i_{M\cap wP}^M\circ \dot{w}\circ r_{w^{-1}Q\cap M}^M(\sigma)$ and $w \neq 1$. By the second adjunction theorem \cite{dhkm_finiteness}, it is isomorphic to
$$\Hom_{B[M \cap wM]}(\dot{w}\circ r_{w^{-1}Q\cap M}^M(\sigma), r_{M\cap w\overline{P}}^M(\sigma)).$$
Now let $a$ be an element of $A_M$. Note that $a \in M \cap wM$ because it even belongs to $A_{M\cap wM}$ via the natural inclusion of groups. As parabolic restriction only occurs on $\sigma_0$, the action of $a$ is:
\begin{itemize}[align=left]
    \item[(LHS)] $r_a \otimes \Psi(a^w)$ where $r_a \in \textup{End}_{A[M \cap wM]}(\dot{w} \circ r_{w^{-1}Q\cap M}^M(\sigma_0))$ is invertible;
    \item[(RHS)] $s_a \otimes \Psi(a)$ where $s_a \in \textup{End}_{A[M \cap wM]}(r_{M \cap w \bar{P}}^M(\sigma_0))$ is invertible.
\end{itemize}
In particular $\alpha  = \Psi(a^w)^{-1} a = \Psi (a^{-w}) a$ acts as $r_a \otimes 1$ on the LHS and as $s_a \otimes \Psi(a a^{-w})$ on the RHS. Because parabolic restriction preserves finitely generated representations, we deduce that $\textup{End}_{A[M \cap wM]}(\dot{w} \circ r_{w^{-1}Q\cap M}^M(\sigma_0))$ is finitely generated over $\mathcal{Z}_{A,M}$ thanks to Lemma \ref{lem:finiteness}. As a result, there exists a polynomial $P_{r_a}$ in $\mathcal{Z}_{A,M}[X]$ that is killing $r_a$. Since $r_a$ is invertible, the Cayley-Hamilton theorem ensures that $r_a$ is killed by a polynomial $P_{r_a}$ of the form $X^n + a_{n-1} X^{n-1} + \dots + a_0$ where $a_0 \in \mathcal{Z}_{A,M}^\times$ and $n \geq 1$.

We concoct a polynomial killing the action of $\alpha$ on the LHS by taking $P_\alpha \in B[X]$ of the form $X^n + (a_{n-1} \otimes 1) \cdot X^{n-1} + \dots + a_0 \otimes 1$. Then $P_\alpha(\alpha)$ acts on the LHS via:
$$r_a^n \otimes 1 + (a_{n-1}r_a^{n-1}) \otimes 1 + \dots + a_0 \otimes 1 = P_{r_a}(r_a) \otimes 1 = 0.$$ 
Let $\beta = P_{\alpha}(\alpha) \in B [A_M]$. By construction, it acts as zero on the LHS. On the RHS however, it acts as $s_a^n \otimes \Psi(a^{-w} a)^n + (a_{n-1} s_a^{n-1}) \otimes \Psi(a^{-w} a)^{n-1} + \dots + (a_0 \otimes 1)$.

We now suppose $a$ chosen so that $a(a^{-1})^{w^{-1}}$ is a non-compact element -- such elements exist by \cite[Eq. (6)]{waldspurger}. We want to show $\{ v \in r_{M\cap w\overline{P}}^M(\sigma) \ | \ \beta \cdot v = 0 \} = \{ 0 \}$. This follows immediately from the next lemma taking $Y = \psi(a^{-w} a)$.

\begin{lemma}
Let $V_0$ be a module over a commutative ring $D$. Let $V = V_0 \otimes_DD[X_1^{\pm 1},\dots, X_r^{\pm 1}]$. Let $Y = X_1^{e_1}\cdots X_r^{e_r} \neq 1$, abbreviated $X^{\vec e}$, and let $\beta(Y) = Y^n + a_{n-1} Y^{n-1} + \dots + a_0 \in D[Y]$ with $n\geq 1$. If $v\in V$ satisfies $\beta(Y)v=0$, then $v=0$.
\end{lemma}
\begin{proof}
Let $v$ be killed by $\beta$ and write:
$$v = \sum_{\vec f}v_{\vec f}X^{\vec f},$$ where $X^{\vec f} = X_1^{f_1}\cdots X_r^{f_r}$. Suppose for contradiction that $v\neq 0$. 

First, suppose that within the set of $\vec f$ such that $v_{\vec f}\neq 0$, there exists at least one element such that $\vec f\cdot \vec e = \sum_{i=1}^r f_i e_i > 0$. In this case, let $F_{max}$ be the set of $\vec f$ such that $\vec f \cdot \vec e$ is maximal and $v_{\vec f}\neq 0$. In particular, note that $\sum_{\vec f\in F_{max}}v_{\vec f} X^{\vec f}\neq 0$. If we write $\beta v = \sum_{\vec g}v'_{\vec g}X^{\vec g}$, let $G_{max}$ be the set of $\vec g$ such that $\vec g\cdot \vec e$ is maximal and $v'_{\vec g}\neq 0$. In particular, since
$$\beta v = \sum_{\vec f}(v_{\vec f}X^{\vec f + n\vec e} + v_{\vec f}a_{n-1}X^{\vec f + (n-1)\vec e} + \cdots + v_{\vec f}a_0 X^{\vec f}),$$
and $\vec e \cdot \vec e>0$, we have $G_{max} = F_{max} + n\vec e$. Since $\beta v=0$ we have $$\sum_{\vec g\in G_{max}}v'_{\vec g}X^{\vec g} = \sum_{\vec f\in F_{max}}v_{\vec f}X^{\vec f + n\vec e} = X^{n\vec e}(\sum_{\vec f\in F_{max}}v_{\vec f}X^{\vec f})=0.$$
As $X^{n \vec e}$ is a unit, this is a contradiction. Thus we conclude that $v=0$ in this case.

Up to multiplying $v$ by a suitable monomial, which is a unit, we can assume the conditions $\vec f \cdot \vec e > 0$ and $v_{\vec f} \neq 0$ hold. Therefore $v=0$.

\end{proof}

We can now conclude that: 
$$\Hom_{B[M \cap wM]}(\dot{w}\circ r_{w^{-1}Q\cap M}^M(\sigma), r_{M\cap w\overline{P}}^M(\sigma))=0.$$
Indeed, let $\phi$ be such a morphism. Since $\phi$ is $B$-equivariant, we must have $\beta \cdot \phi(w) = \phi (\beta \cdot w) = 0$ for all $w$ in the LHS. But we have just proved that only the zero element is killed by $\beta$ on the RHS. Therefore, for all $w$, we have $\phi(w) = 0$, \textit{i.e.} $\phi=0$. This concludes the proof of Theorem~\ref{thm:intertwining}.
\end{proof}

Next we turn to proving Theorem~\ref{thm:image}. Before we begin, we establish two lemmas.

\begin{lemma}\label{lemma:finiteness_polynomial}
Suppose $\sigma_0$ is a finitely generated $A[M]$-module. For any $w$ and any element $a$ in $A_{wM\cap M}$, let $\phi_w(a)$ denote the endomorphism of $r^M_{M\cap w^{-1}Q}(\sigma)$ induced by the action of $w^{-1}aw$ under right translation. There exists a polynomial 
$$f_{w,a}(X) = X^d + b_{d-1}X^{d-1} + \cdots + b_1X + b_0 \text{\ \  in  $\cZ_{A,M}[X]$}$$ with $b_0\in \cZ_{A,M}^{\times}$ such that the element $f_{w,a}(\phi_w(a))$ is the zero endomorphism.
\end{lemma}
\begin{proof}
Since $\sigma_0$ is admissible over $\cZ_{A,M}$, and parabolic restriction preserves admissibility, the endomorphism rings $\End_{\cZ_{A,M}[M]}(r^M_{M\cap w^{-1}Q}(\sigma_0)$ is finitely generated as a $\cZ_{A,M}$-module. By the Cayley--Hamilton theorem, there exists a monic polynomial $f_{w,a}(X)$ in $\cZ_{A,M}[X]$ such that $f_{w,a}(\phi_w(a))=0$ in $\End_{\cZ_{A,M}[M]}(r^M_{M\cap w^{-1}Q}(\sigma_0)))$. Since $\phi_w(a)$ is coming from the action of a group element, it is an invertible endomorphism, so the constant term $b_0$ of the characteristic polynomial is a unit.
\end{proof}

\begin{lemma}\label{lemma:twisted_finiteness_polynomial}
For any $w$ and any element $a$ in $A_{wM\cap M}$, let $\Phi_w(a)$ denote the endomorphism of $r^M_{M\cap w^{-1}Q}(\sigma)$ induced by the action of $w^{-1}aw$ via right translation. In the notation of Lemma~\ref{lemma:finiteness_polynomial}, the polynomial 
$$F_{w,a}(X)=(1\otimes \overline{w^{-1}aw}^{-d})X^d + (b_{d-1}\otimes \overline{w^{-1}aw}^{-(d-1)})X^{d-1} + \cdots + (b_1\otimes \overline{w^{-1}aw}^{-1})X + b_0\otimes 1$$ in $\fR[X]$ has the property that $F_{w,a}(\Phi_w(a))$ is the zero endomorphism.
\end{lemma}
\begin{proof}
Let $h\otimes 1$ be a simple tensor in $\sigma = \sigma_0\otimes_{\ZZ'} \ZZ'[M/M^0]$ and let $\phi_w(a)$ be the endomorphism defined in Lemma~\ref{lemma:finiteness_polynomial}. To ease notation we abbreviate $\Phi_w=\Phi_w(a)$ and $\phi_w = \phi_w(a)$.
\begin{align*}
F_{w,a}(\Phi_w)(h\otimes 1) &= \sum_{k=0}^d(b_k\otimes \overline{w^{-1}aw}^{-k})\cdot \Phi_w^k((h\otimes 1))\\
&= \sum_{k=0}^d(b_k\otimes \overline{w^{-1}aw}^{-k})\cdot (\phi_w^k h\otimes \overline{w^{-1}aw}^k)\\
&= \sum_{k=0}^d (b_k\phi_w^k h)\otimes 1\\
&= \left(\sum_{k=0}^d b_k\phi_w^k h\right)\otimes 1\\
&= (f_{w,a}(\phi_w)h)\otimes 1 = 0\otimes 1
\end{align*}
\end{proof}

\begin{proof}[Proof of Theorem~\ref{thm:image}]

Let $E_w$ denote the endomorphism algebra $\End_{\fR[wM\cap M]}(\dot{w}\circ r_{Q\cap w^{-1}(M)}^M(\sigma))$. Given an element $u$ of $\fR[A_{wM \cap M}]$, we will let $\rho_w(u)$ denote the element of $E_w$ induced by the functor $\dot{w}\circ r_{Q\cap w^{-1}(M)}^M$. Observe that under the natural ring isomorphism
$$E_w \cong \End_{\fR[M\cap w^{-1}M]}(r_{Q\cap w^{-1}(M)}^M(\sigma)),$$ the endomorphism $\rho_w(a)$, $a\in A_M$, corresponds to the endomorphism $\Phi_w(a)$ appearing in Lemma~\ref{lemma:twisted_finiteness_polynomial}. 

For each $w\neq 1$, let $a_w$ be an element of $A_M$ such that $a_ww^{-1}a_w^{-1}w$ is not in $M^0$ (see Equation (6) in [Waldspurger]). By viewing $a_w$ as an element of $A_{wM\cap M}$ by way of the natural inclusion $A_M\to A_{wM\cap M}$, we can apply Lemma~\ref{lemma:twisted_finiteness_polynomial} to obtain polynomials $F_{w,a_w}(X)$ in $\fR[X]$ such that $F_{w,a_w}(\rho_w(a_w))=0$ in $E_w$.

Now define the element $\beta$ of the group ring $\fR[A_M]$ by $$\beta := \prod_{w\neq 1}F_{w,a_w}(a_w).$$ Since $F_{w,a_w}(a_w)$ acts by zero on each $\dot{w}\circ r_{w^{-1}Q\cap M}^M(\sigma)$, $w\neq 1$, it also acts by zero on each subquotient $i_{M\cap wP}^M \circ \dot{w}\circ r^M_{w^{-1}Q\cap M}(\sigma)$ of the geometric lemma filtration. Thus $\beta$ defines an element of the ideal $\fI_\sigma^{QP}\subset \fR[A_M]$. In Subsection~\ref{sec:localizing} we defined the element $J_{\beta}$ in
$\Hom_{\fR[M]}\left(r_Q^G\circ i_P^G (\sigma) ,\sigma\right)$ induced by $\beta$. We now describe an element $b$ of $\fS\subset \fR$ such that $\sig_{PQ}(J_\beta)$ coincides with $b$ in $\End_{\fR[M]}(\sigma)$.

Recall our notation $E_1= \End_{\fR[M]}(\sigma)$, and for an arbitrary element $u$ of $\fR[A_M]$, recall that $\rho_1(u)$ denotes the element of $E_1$ defined by right-translation. In particular $\rho_1(u)$ lands in the image of the natural map $$\fR\to \End_{\fR[M]}(\sigma).$$ Given $a$ in $A_M$, we can describe $\rho_1(a)$ as an element of $\fR$ in the following manner:
\begin{itemize}
\item since $a$ is in $A_M$ it defines an endomorphism of the identity functor and thus an element of $\cZ_{A,M}$, which we will denote $z_{a}$;
\item since $a$ is in $M$ it defines an element of $A[M/M^0]$, which we will denote $\overline{a}$;
\item $\rho_1(a)$ is the element $z_a\otimes \overline{a}$ of $\fR = \cZ_{A,M}\otimes A[M/M^0]$.
\end{itemize}
In the above notation, we have
$$
\rho_1(F_{w,a_w}(a_w)) = F_{w,a_w}(z_{a_w}\otimes \overline{a_w}),$$
which leads us to define.
$$b:=\rho_1(\beta) = \prod_{w\neq 1}F_{w,a_w}(z_{a_w}\otimes \overline{a_w})\ .$$
Using the description of $F_{w,a_w}(X)$ in Lemma~\ref{lemma:twisted_finiteness_polynomial}, each of the factors $F_{w,a_w}(z_{a_w}\otimes \overline{a_w})$ can be written
\begin{align*}
\sum_{k=0}^d(b_k\otimes \overline{w^{-1}a_ww}^{-k})(z_{a_w}\otimes \overline{a_w})^k= \sum_{k=0}^db_kz_{a_w}^k\otimes (\overline{a_ww^{-1}a_w^{-1}w})^k\ ,
\end{align*}
where $b_d=1$ and $b_0$ and $z_{a_w}$ are in $\cZ_{A,M}^{\times}$. Thus $F_{w,a_w}(z_{a_w}\otimes \overline{a_w})$ is in $\fS$ and so $b$ is also in $\fS$.

\end{proof}

\section{Schur's lemma}\label{sec:schur}
In this section we use the results of Section~\ref{sec:unramifiedtwisting} to deduce that universal unramified twisting and parabolic induction preserves Schur's lemma. Recall the notation $R = A[M/M^0]$ and $\fR = \cZ_{A,M}[M/M^0]$.

\begin{lemma}
\label{lem:equivariance} \
\begin{enumerate}
\item Let $\sigma_0$ be an arbitrary smooth $A[M]$-module. Then $\End_{\fR[M]}(\sigma) = \End_{R[M]}(\sigma)$. 
\item If the image of $\cZ_{A,M}$ in $\End_{A[M]}(\sigma_0)$ is contained in the image of $A\to \End_{A[M]}(\sigma_0)$, then $$\Hom_{\fR[G]}(i_P^G(\sigma),i_Q^G(\sigma))=\Hom_{R[G]}(i_P^G(\sigma),i_Q^G(\sigma)).$$
\end{enumerate}
\end{lemma}

\begin{proof}
(i) Let $\phi:\sigma\to \sigma$ be $R[M]$-equivariant; we wish to show it must also be $\fR$-equivariant. Let $\sigma'$ be the $R[M]$-module whose $R$-module structure is the same as that of $\sigma$, but where the $M$-action is given by $\sigma'(m)(v\otimes 1) = \sigma_0(m)\otimes 1$. Then $\phi$ defines an $R$-module endomorphism $\sigma'\to \sigma'$. Given $v\otimes 1\in \sigma'$ and $m\in M$, we have 
\begin{align*}
\phi(\sigma'(m)(v\otimes 1)) &=\phi(\overline{m}^{-1}\sigma(m)(v\otimes 1))\\
&=\overline{m}^{-1}\sigma(m)\phi(v\otimes 1)\text{\ \ \ (by $R[M]$-equivariance)}\\
&=\sigma'(m)\phi(v\otimes 1)\ ,
\end{align*}
hence $\phi$ defines an $R[M]$-module homomorphism $\sigma'\to \sigma'$. Let $\cZ_{R,M}$ be the center of the category $\Rep_R(M)$. There is a natural ring homomorphism $\fR\to \cZ_{R,M}$, through which the action of $\fR$ on $\sigma'$ factors. Because $\phi:\sigma'\to \sigma'$ is a morphism in $\Rep_R(M)$, it must commute with the natural action of $\fR$ on $\sigma'$. But the action of $\fR$ on $\sigma'$ is the same as action of $\fR$ on $\sigma$, by definition.

(ii) The action of $\fR$ on $i_P^G(\sigma)$ is as follows. Take $f:G\to \sigma$ an arbitrary element of $i_P^G(\sigma)$, $z\in \cZ_{A,M}$, and $m\in M$, the action of $z\overline{m}\in \fR$ is $(z\overline{m}\cdot f)(g) = (z\overline{m})f(g)$. Let $\lambda\in A$ be the scalar by which $z$ acts on $\sigma_0$, so that $(z\overline{m})f(g) = \lambda\overline{m}f(g)$. Since any element $\phi\in \Hom_{R[M]}(i_P^G(\sigma),i_Q^G(\sigma))$ commutes with scaling by $\lambda\overline{m}\in R$, the statement follows.
\end{proof}

Suppose $\sigma_0$ is finitely generated. We claim $\textup{End}_{R[M]}(\sigma) \cong \textup{End}_{A[M]}(\sigma_0) \otimes_A R$ and simply sketch the proof because this is a variation of Proposition~\ref{prop:free-extension-scalars-hom-fg} since the $M$-actions on $\sigma = \sigma_0 \otimes_A \Psi$ and $\sigma_0 \otimes_A R$ are not quite the same. We consider the embedding $i : v_0 \in \sigma_0 \mapsto v_0 \otimes_A 1_R \in \sigma$. It is not $M$-equivariant, but because $R$ is free over $A$, the restriction to $\sigma_0$ gives an isomorphism
$$\varphi \in \textup{End}_{R[M]}(\sigma) \mapsto \varphi \circ i \in \textup{Hom}_{A[M]}(\sigma_0, \sigma_0 \otimes_A R).$$
Indeed $\varphi(i(\sigma_0(m) v_0)) = \bar{m}^{-1} \varphi( \sigma(m) (v_0 \otimes_A 1_R)) = (\sigma_0(m) \otimes_A 1_R) \varphi(i(v_0))$. We prove the right-hand side is isomorphic to $\textup{End}_{A[M]}(\sigma_0) \otimes_A R$ as in Proposition~\ref{prop:free-extension-scalars-hom-fg} and therefore the natural ring morphism $\textup{End}_{A[M]}(\sigma_0) \otimes_A R \to \textup{End}_{R[M]}(\sigma)$ must be an isomorphism.

If we assume the natural map $A\to \End_{A[M]}(\sigma_0)$ is an isomorphism, then $R\to \End_{R[M]}(\sigma)$ is also an isomorphism. In this context, Corollary~\ref{cor:sig_injective} tells us that taking the $(\sigma,P,Q)$-signature (together with Frobenius reciprocity) defines an injection of $R$-modules  $$\sig_{PQ}:\Hom_{R[G]}(i_P^G(\sigma),i_Q^G(\sigma)) \hookrightarrow \End_{R[M]}(\sigma)=R.$$ The image of this map defines an ideal of $R$.

When we take $Q=P$, we deduce as in Corollary~\ref{cor:schurslemmafrakR} that $J_{P|P}$ is the identity and Schur's lemma holds for the family $i_P^G(\sigma)$ (without any localization):
\begin{corollary}\label{cor:genericschur}
Suppose $\sigma_0$ is finitely generated and $A \to \End_{A[M]}(\sigma_0)$ is an isomorphism. Then the natural map $R \to \End_{R[G]}(i_P^G(\sigma))$ is an isomorphism.
\end{corollary}

As in Section~\ref{sec:localizing}, the ideal
$$I_\sigma^{QP} = \ker\bigg(R[A_M]\to \End_{R[M]}((r_Q^Gi_P^G/F_{QP}^{<1})(\sigma))\bigg)$$
provides a source of intertwining operators; given $\beta$ in $I_\sigma^{QP}$ we denote by $J_{\beta}$ the corresponding intertwining operator.

\begin{definition}
Let $S$ be the multiplicative subset of $R = A[M/M^0]$ generated by elements of the form $$c_d\overline{m}^d + c_{d-1}\overline{m}^{d-1} + \cdots + c_1\overline{m}+c_0,\ \ m\in M\cap G^0,\ c_i\in A$$ such that both $c_d$ and $c_0$ are units in $A$.
\end{definition}

As in Section \ref{sec:localizing} we deduce the following:

\begin{corollary}\label{cor:rationalintertwining}
Suppose $\sigma_0$ is finitely generated and $A\to\End_{A[M]}(\sigma_0)$ is an isomorphism. There exists $\beta\in I_\sigma^{QP}$ such that $\sig_{QP}(J_{\beta})$ acts on $\sigma$ by a scalar $b$ in the multiplicative set $S$. In particular, we get an isomorphism of $R[1/b]$-modules
$$\begin{array}{ccl} \Hom_{R[1/b][G]}(i_P^G(\sigma[1/b]),i_Q^G(\sigma[1/b])) & \overset{\sim}{\leftrightarrow} & R[1/b] \\
J& \mapsto & \sig_{PQ}(J)\\
r \cdot J_{Q|P}(\sigma_0) & \mapsfrom & r 
\end{array}$$
\end{corollary}

\section{The Harish-Chandra $j$-function}\label{sec:jfunction}
For a finitely generated $A[M]$-module $\sigma_0$, we define the Harish-Chandra $j$-function as follows:
$$j_P^G(\sigma_0) = J_{\overline{P}|P}(\sigma_0)\circ J_{P|\overline{P}}(\sigma_0)\in \End_{\fS^{-1}\fR[G]}(i_P^G(\fS^{-1}\sigma)) \overset{\text{Cor~\ref{cor:schurslemmafrakR}}}{\cong} \End
_{\fS^{-1}\fR[M]}(\fS^{-1}\sigma).$$ Note that $j_P^G(\sigma_0)$ could also be defined in terms of $\sigma[1/b]$ over the ring $\fR[1/b]$ where $b$ is the product $b_{\overline{P}P}b_{P\overline{P}}$ of elements $b_{\overline{P}P}$ and $b_{P\overline{P}}$ chosen to define $J_{\overline{P}|P}(\sigma_0)$ and $J_{P|\overline{P}}(\sigma_0)$, respectively. If we suppose in addition that the map $A\to \End_{A[M]}(\sigma_0)$ is an isomorphism, then $j_P^G(\sigma_0)$ is given by a scalar in $R[1/b]\subset S^{-1}R$ by Corollary~\ref{cor:genericschur}. 

\begin{rem}
Both $J_{Q|P}(\sigma_0)$ and $j(\sigma_0)$, as we have defined them, depend on choices of Haar measures on $U_{\overline{P}}$ and $U_P$.
\end{rem}

\begin{prop}\label{prop:centrality}
$j_P^G(\sigma_0)$ lies in the center of $\End_{\fS^{-1}\fR[G]}(i_P^G(\fS^{-1}\sigma))$.
\end{prop}
\begin{proof}
We will show the image of $j_P^G(\sigma_0)$ in the composite $$\End_{\fS^{-1}\fR[G]}(i_P^G(\fS^{-1}\sigma))\cong \Hom_{\fS^{-1}\fR[M]}(r_P^Gi_P^G(\fS^{-1}\sigma),\fS^{-1}\sigma)\xrightarrow{\sig_{PP}} \End_{\fS^{-1}\fR[M]}(\fS^{-1}\sigma)$$ of Corollary~\ref{cor:schurslemmafrakR} lands in the center of $\End_{\fS^{-1}\fR[M]}(\fS^{-1}\sigma)$. The image of $j_P^G(\sigma_0)$ under the first isomorphism (Frobenius reciprocity) is the map $$\bar{j}:r_P^Gi_P^G(\fS^{-1}\sigma)\to \sigma$$ induced by $f\mapsto j_P^G(\sigma_0)(f)(1)$ for $f\in i_P^G(\sigma)$; we need to show $\sig_{PP}(\bar{j})$ is central. 

The ring $\End_{\fS^{-1}\fR[M]}(\sigma)$ acts on $i_P^G(\sigma)$ via $(\phi\cdot f)(g) = \phi(f(g))$, $\phi\in \End(\sigma)$, $f\in i_P^G(\sigma)$, and also on $r_P^Gi_P^G(\sigma)$ by functoriality of parabolic restriction. This induces a right action of $\End_{\fS^{-1}\fR[M]}(\sigma)$ on $\Hom_{\fS^{-1}\fR[M]}(r_P^Gi_P^G(\sigma),\sigma)$ by precomposition. By construction, $\sig_{PP}$ (in fact, $\sig_{PQ}$ for any $Q$) is equivariant for this action, namely
$$\sig_{PP}(J\cdot\phi) = \sig_{PP}(J)\phi$$ for $\phi\in \End_{\fS^{-1}\fR[M]}(\sigma)$,  $J\in \Hom_{\fS^{-1}\fR[M]}(r_P^Gi_P^G(\sigma),\sigma),$ so it suffices to prove $\bar{j}\cdot \phi = \phi\cdot \bar{j}$ for $\phi\in \End_{\fS^{-1}\fR[M]}(\sigma)$.

Let $b\in \fR$ and $\beta\in \fR[A_M]$ (respectively, $b'$ and $\beta'$) be the elements chosen to define $J_{\overline{P}|P}(\sigma_0)$ (respectively, $J_{P|\overline{P}}$), so we have
$$\bar{j}(\bar{f}) = j_P^G(\sigma_0)(f)(1) = \frac{1}{b'}\int_{U_P}\beta'\left(\frac{1}{b}\int_{U_{\overline{P}}}(\beta f)(uu')du\right)du'.$$ Then since $\beta$ and $\beta'$ are in the group ring of the center of $M$ and $b$, $b'$ are scalars, we can check directly:
\begin{align*}
(\bar{j}\cdot \phi)(\bar{f}) &= \bar{j}(\phi\bar{f})\\
&= \frac{1}{b'}\int_{U_P}\beta'\left(\frac{1}{b}\int_{U_{\overline{P}}}(\beta \phi f)(uu')du\right)du'\\
&= \phi\left(\frac{1}{b'}\int_{U_P}\beta'\left(\frac{1}{b}\int_{U_{\overline{P}}}(\beta f)(uu')du\right)du'\right)\\
&= \phi(\bar{j}(\bar{f}))\\
&= (\phi\cdot \bar{j})(\bar{f})
\end{align*}

\end{proof}

We proceed in a manner similar to \cite[Appendix C]{dhkm_conjecture} to study the properties of the $j$-function.

The argument of \cite[IV.3(1)]{waldspurger} shows that, if $P'$ is standard parabolic subgroup, $$j_{P'}^G(\sigma_0)J_{P'|P}(\sigma_0) = j_P^G(\sigma_0)J_{P'|P}(\sigma_0).$$ Taking the $(\sigma,P,P')$ signature of both sides, we get \begin{align*}
\sig_{PP'}\left(j_{P'}^G(\sigma_0)J_{P'|P}(\sigma_0)\right) & = \sig_{PP'}\left(j_P^G(\sigma_0)J_{P'|P}(\sigma_0)\right)\\
j_{P'}^G(\sigma_0)\sig_{PP'}\left(J_{P'|P}(\sigma_0)\right) & = j_P^G(\sigma_0)\sig_{PP'}\left(J_{P'|P}(\sigma_0)\right)\\
j_{P'}^G(\sigma_0)\cdot \text{id} &= j_{P}^G(\sigma_0)\cdot \text{id}\ ,
\end{align*}
so $j_P^G(\sigma_0)$ is independent of the choice of standard parabolic $P$. As such, we can drop the subscript $P$ from our notation. If $G$ is clear from the context, we sometimes omit the superscript $G$.

From Proposition~\ref{prop:compatibility} and \ref{prop:bernsteinscalarextension}, we deduce the following:
\begin{prop}\label{prop:compatibility_of_j} Let $\sigma_0$ be a finitely generated $A[M]$-module. In the notation of Section~\ref{sec:compatibility_of_intertwiningoperators}, we have
\begin{enumerate}
\item Let $q_0:\sigma_0\to \sigma_0'$  be a homomorphism of finitely generated $A[M]$-modules and $q:\sigma\to \sigma'$ be the corresponding homomorphism of $\fR[M]$-modules. Then $j(\sigma_0)\in Z(\End_{\fR[M]}(\fS^{-1}\sigma))$ stabilizes the kernel of the homomorphism $\fS^{-1}\sigma\to \fS^{-1}\sigma'$ and the endomorphism it induces on the image agrees with $j(\sigma_0')$.
\item Given $A\to A'$ any homomorphism, $$j(\sigma_0\otimes_AA') = j(\sigma_0)\otimes \textup{id}\in Z(\End_{\fR'[M]}({\fS'}^{-1}(\sigma\otimes_RR'))).$$
\item Given $\lambda:\cZ_{A,M}\to B$ and the induced morphism $\widetilde{\lambda}:\fS^{-1}\fR \to \cS^{-1}\cB$, we have
$$j(\sigma_0\otimes_{\cZ_{A,M},\lambda}B) = j(\sigma_0)\otimes\textup{id}\in Z(\End_{\cS^{-1}\cB[M]}(\cS^{-1}(\sigma\otimes_{\fR,\widetilde{\lambda}}\cB))).$$
\end{enumerate}
\end{prop}

\begin{rem}\label{rem:drop_standard} As in \cite[App C.1]{dhkm_conjecture}, we can also deduce that $J_{Q|P}(\sigma_0)$, and thus $j(\sigma_0)$, are compatible with isomorphisms $G\to G'$ of connected reductive $F$ groups, in a sense that is made precise in loc. cit. It follows that the requirement that $P$ is semistandard can be relaxed, taking into account this notion of compatibility.
\end{rem}

We can give a more formal alternative proof of Proposition~\ref{prop:centrality} by interpreting $j(-)$ as an endomorphism of functors.
Recall the functor $\fri_P^G$ defined in Subsection~\ref{sec:compatibility_of_intertwiningoperators}: it is the composition of the functors $\fri:\Rep_A^{\text{f.g.}}(M)\to \Rep_{\fS^{-1}\fR}(M)$ and $i_P^G : \Rep_{\fS^{-1}\fR}(M) \to \Rep_{\fS^{-1}\fR}(G)$, where $\fri$ is given by $\sigma_0\mapsto \fS^{-1}\sigma$. Proposition~\ref{prop:compatibility_of_j} can be phrased as saying that $\sigma_0\mapsto j(\sigma_0)$ defines an endomorphism of the functor $\fri_P^G$. In particular, $j(\sigma_0)$ commutes with elements of the subring $$\fri_P^G(\textup{End}_{A[M]}(\sigma_0))\subset \End_{\fS^{-1}\fR[G]}(i_P^G(\fS^{-1}\sigma)).$$ By Corollary~\ref{cor:schurslemmafrakR}, the latter is isomorphic to $\End_{\fS^{-1}\fR[M]}(\fS^{-1}\sigma)$, and thus $j(\sigma_0)$ commutes with the subring $\fri(\textup{End}_{A[M]}(\sigma_0))$. But the latter ring generates $\End_{\fS^{-1}\fR[M]}(\fS^{-1}\sigma)$ as a $\fS^{-1}\fR$-algebra. Indeed, by flatness of localization and Proposition \ref{extension_of_scalars_flat_lam_prop}, it is enough to show that $\textup{End}_{A[M]}(\sigma_0)$ generates $\textup{End}_{\mathfrak{R}[M]}(\sigma)$ as an $\mathfrak{R}$-algebra, but Lemma \ref{lem:equivariance} and the discussion thereafter show that $\textup{End}_{\mathfrak{R}[M]}(\sigma) = \textup{End}_A(\sigma_0) \otimes_A R$. We deduce that $j(\sigma_0)$ is central.

Note that if we suppose in addition that $A\to \End_{A[M]}(\sigma_0)$ is an isomorphism, the statements in Proposition~\ref{prop:compatibility_of_j} become, respectively, 
\begin{enumerate}
\item $j(\sigma_0)=j(\sigma_0')$ in $S^{-1}R$,
\item $j(\sigma_0\otimes_AA')$ is the image of $j(\sigma_0)$ under $S^{-1}f:S^{-1}R\to S^{-1}R'$,
\item $j(\sigma_0\otimes_{\cZ_{A,M},\lambda}B)$ is the image of $j(\sigma_0)$ under $\widetilde{\lambda}:\fS^{-1}\fR\to \cS^{-1}\cB$.
\end{enumerate}
In particular, $j(\sigma_0)$ is an invariant of the isomorphism class of $\sigma_0$ when $\sigma_0$ satisfies Schur's lemma. When it doesn't, then given an isomorphism $\sigma_0\cong\sigma_0'$, $j(\sigma_0)$ and $j(\sigma_0')$ will differ by the corresponding isomorphism $\End_{\fS^{-1}\fR[M]}(\fS^{-1}\sigma)\cong \End_{\fS^{-1}\fR[M]}(\fS^{-1}\sigma')$.

\subsection{Invariance by unramified characters of $G$.} Let $\Psi_M = \textup{Spec}(A[M/M^0])$. The group scheme acts on $\textup{Rep}_A(M)$ in the following way. First of all, when $B$ is an $A$-algebra, we consider the functor $V \in \textup{Rep}_A(M) \mapsto V \otimes_A B \in \textup{Rep}_B(M)$ called the extension of scalars. Now, we define an action via the functor of points perspective \textit{i.e.} for any $A$-algebra $B$ we define it on $B$-valued points
$$\begin{array}{ccc}
    \Psi_M(B) \times \textup{Rep}_B(M) & \to & \textup{Rep}_B(M) \\
    (\psi,(\sigma,V)) & \mapsto & (\sigma^\psi,V^\psi)
    \end{array}$$
where $V^\psi = V$ and the $M$-action is given by $\sigma^\psi(m) \cdot v = \psi(m) \sigma(m) \cdot v$ for $m \in M$ and $v \in V$. We may sometimes denote $V^\psi$ by $V \otimes_B \psi$.

We say that $V \in \textup{Rep}_A(M)$ is $\Psi_M$-invariant if, for all $A$-algebras $B$ and all $\psi \in \Psi_M(B)$, the $B[M]$-modules $(V \otimes_A B)^\psi$ and $V \otimes_A B$ are isomorphic. We can also look at the action, obtained by precomposing, of $\Psi_M$ on $A$-linear functors such as $J_{Q|P}$ from $\textup{Rep}_A(M)$ to an $A$-linear category $\mathcal{C}$. Since we have an exact sequence $1 \to (M \cap G^0)/M^0 \to M/M^0 \to G/G^0 \to 1$, the unramified characters of $G$ form a closed group subscheme $\Psi_G \subset \Psi_M$.

\begin{prop} The functor $J_{Q|P}$ is $\Psi_G$-invariant. \end{prop}

\begin{proof} Let $B$ be an $A$-algebra and $\psi \in \Psi_G(B)$. We want to show $J_{Q|P}(\sigma_0^{\psi|_M})$ and $J_{Q|P}(\sigma_0)$ define the same morphism. We have to make it precise what ``same'' means in this situation.

Note that $i_P^G(\sigma^{\psi|_M}) \cong i_P^G(\sigma)^\psi$ given by $f \mapsto f^\psi$ in $\textup{Rep}_B(G)$ where $f^\psi(g) = \psi(g)^{-1} f(g)$. As a function, $f^\psi$ belongs to the $B$-module $i_P^G(\sigma)$. The element of $(\sigma,P,Q)$-signature identity is not affected by the twist by an unramified character of $M$. Another way to phrase is that $J_{Q|P}(\sigma_0^{\psi|_M})(f)$ is identified with $J_{Q|P}(\sigma_0)(f^\psi)$ by the characterization in Corollary~\ref{cor:characterization}. This allows us to identify $J_{Q|P}(\sigma_0^{\psi|_M})$ and $J_{Q|P}(\sigma_0)^\psi$.

Because $i_P^G(\sigma)^\psi = i_P^G(\sigma)$ as $B$-modules, this allows us to compare $J_{Q|P}(\sigma_0)^\psi$ and $J_{Q|P}(\sigma_0)$. They are in fact equal because the functor defined by $\psi$ is the identity on morphisms \textit{i.e.} the functor $V \mapsto V^\psi$ associates to $\phi \in \textup{Hom}_{B[G]}(V,W)$ the morphism $\phi^\psi=\phi$. \end{proof}

As a consequence of this invariance of $J_{Q|P}$, it becomes straightforward that:

\begin{corollary}\label{cor:G^0}
Suppose $\sigma_0$ satisfies Schur's lemma. The element $j^G(\sigma_0)\in S^{-1}R$ lies in the subring $S^{-1}R^G$ where $R^G = A[(M\cap G^0)/M^0]$ (recall that $S\subset R^G$ by definition).
\end{corollary}

\subsection{Factorization}
Suppose $P$ is a semistandard parabolic. Fix an ordering $\Sigma_{red}=\{\alpha_1,\dots, \alpha_r\}$ of the reduced roots of $A_M$ in $P$ such that there are sequences $(P_0,\dots, P_r)$ and $(Q_0,\dots, Q_r)$ of semistandard parabolic subgroups satisfying:
\begin{itemize}
\item $P_0=P$ and $P_r = \overline{P}$ and each $P_i$ has Levi component $M$,
\item $\Sigma_{red}(P_i)\cap \Sigma_{red}(\overline{P_{i-1}}) = \alpha_i$,
\item $Q_i$ has Levi $M_{\alpha_i}$, the Levi subgroup generated by $M$ and the root subgroup attached to $\alpha_i$.
\end{itemize}
In this situation, we have $\overline{P_i\cap M_{\alpha_i}} = P_{i-1}\cap M_{\alpha_i}$, along with the properties $d(P_i,P_{i+1}) = 1$ and $d(P,\overline{P}) = \sum_id(P_i, P_{i+1})$.

For each $i=1,\dots, r$ we have the $j$-function relative to $M_{\alpha_i}$
$$j^{M_{\alpha_i}}(\sigma_0)\in \End_{\fS^{-1}\fR[M_{\alpha_i}]}(i_{P_i\cap M_{\alpha_i}}^{M_{\alpha_i}}(\fS^{-1}\sigma))\cong \End_{\fS^{-1}\fR[M]}(\fS^{-1}\sigma).$$
By combining Lemmas~\ref{lem:factorization_of_intertwiners} and \ref{lem:multiplicativity_of_intertwiners}(i), we obtain:

\begin{prop}\label{prop:factorization_of_j}
Let $\sigma_0$ be finitely generated. In the ring  $\End_{\fS^{-1}\fR[M]}(\fS^{-1}\sigma)$, $j^G(\sigma_0)$ factors as $$j^G(\sigma_0) = \prod_{i=1}^r j^{M_{\alpha_i}}(\sigma_0).$$
\end{prop}
All the factors are central, so the product on the right side is independent of the ordering of the factors. When $\sigma_0$ satisfies Schur's lemma, the terms in the product are elements of the commutative ring $S^{-1}R$, so this independence of ordering is trivial.

\begin{rem}
The $j$-function of a parabolically induced representation enjoys a multiplicativity property, whose statement and proof is analogous to that appearing in \cite[Appendix C.3]{dhkm_conjecture}. We omit it.
\end{rem}

\section{More on $b$}
Let $\sigma_0$ be finitely generated satisfying Schur's lemma. By Corollary~\ref{cor:G^0}, each term $j^{M_{\alpha_i}}(\sigma_0)$ in Proposition~\ref{prop:factorization_of_j} lives in the subring $S_i^{-1}R_i$ where $R_i = A[(M\cap M_{\alpha_i}^0)/M^0]$ and $S_i$ is the multiplicative subset of $R_i$ defined in Section~\ref{sec:schur}, where $G$ is taken to be $M_{\alpha_i}$. The group $(M\cap M_{\alpha_i}^0)/M^0$ is a free abelian group of rank 1 and (as observed in \cite[Sec 7]{dat}) it has a unique generator $m_{\alpha_i}$ with the property that there is a positive power $n$ with $m_{\alpha_i}^n\in A_M$ and $|\alpha_i(m_{\alpha_i}^n)|\leq 1$. This gives us a canonical isomorphism
\begin{align*}
R_i &\to A[X_i^{\pm 1}]\\
m_{\alpha_i}&\mapsto X_i
\end{align*}
that sends $S_i$ to the subset of polynomials whose first and last coefficients are units. Now if we identify $A[M/M^0]\cong A[X_1^{\pm 1},\dots, X_r^{\pm 1}]$ it follows from Proposition~\ref{prop:factorization_of_j} and Corollary~\ref{cor:G^0} that $j^G(\sigma_0)$ can be expressed as a product of rational functions, each in a single variable:
$$j^G(\sigma_0) = \prod_i\frac{P_i(X_i)}{Q_i(X_i)}\in S^{-1}(A[X_1^{\pm 1},\dots, X_r^{\pm 1}])$$ with $P_i(X_i), Q_i(X_i)$ in $A[X_i^{\pm 1}]$, and $Q(X_i)$ having first and last coefficients units in $A$.

In the multiplicative sets $S_i\subset A[X_i^{\pm 1}]$ of polynomials whose first and last coefficients are units, we can speak of the degree of a polynomial $$\deg(a_nX^n+a_{n+1}X^{n+1}+\cdots+ a_mX^m):= m-n,$$ and multiplying polynomials in $S_i$ causes degrees to add even if $A$ is not an integral domain. Thus if we want to invert as little as possible in defining $J_{Q|P}(\sigma_0)$ we can, for each $i$, take $G$ in Theorem~\ref{thm:image} to be $M_{\alpha_i}$ and choose each $\beta_i$ such that $\deg_{X_i}(b_i)$ is as small as possible. This provides a denominator $b = b_1\cdots b_r$ that is minimal with respect to the divisibility partial ordering, and the ring $R[1/b]$ is minimal with respect to inclusion. 

If $A$ is a field, $A[X_i^{\pm 1}]$ is a principal ideal domain. Since image of $\sig_{PQ}$ defines an ideal, which is generated by a single element, there is a canonical choice of $b_i$ having least degree. When $A$ is not a field, it is not clear there is a canonical denominator.

\section{Universal $j$-function}\label{sec:universal}

Fix a depth $r\geq 0$ and let $\cP_0$ be the projective generator of the depth $\leq r$ subcategory of smooth $\ZZ'[M]$-modules (see \cite[Appendice A]{dat_finitude} for a description of $\cP_0$). In this section, we will apply the construction of Sections~\ref{sec:unramifiedtwisting} and  \ref{sec:schur} with $A = \ZZ'$ and $\sigma_0 = \cP_0$. In the notation of Sections~\ref{sec:unramifiedtwisting} and  \ref{sec:schur} we have
\begin{align*}
R &= \ZZ'[M/M^0]\\
\sigma &= \cP:=\cP_0\otimes_{\ZZ'}\Psi_M\\
\fR &= \cZ_{\ZZ',M}[M/M^0]
\end{align*}
where $\cP_0$ is finitely generated as a $\ZZ'[M]$-module. As $\cP_0$ is a projective generator, the map $\cZ_{\ZZ',M}\to Z(\End_{\ZZ'[M]}(\cP_0))$ is an isomorphism on the depth $\leq r$ direct factor $\cZ_{\ZZ',M,r}$ of $\cZ_{\ZZ',M}$ and zero on the other factors. In particular, so is $\cZ_{\ZZ',M,r}\to Z(\End_{\cZ_{\ZZ',M}[M]}(\cP_0))$ because $\End_{\ZZ'[M]}(\cP_0)= \End_{\cZ_{\ZZ',M}[M]}(\cP_0)$. Finally, passing to universal unramified twists we have an isomorphism
$$\fR\overset{\sim}{\to} Z(\End_{\fR[M]}(\cP)),$$ which equals $Z(\End_{R[M]}(\cP))$ by Lemma~\ref{lem:equivariance}(i).

The constructions of Sections~\ref{sec:unramifiedtwisting},  \ref{sec:schur}, and \ref{sec:jfunction} give
\begin{align*}
J_{Q|P}(\cP_0) &\in \Hom_{\fS^{-1}\fR[G]}(i_P^G(\fS^{-1}\cP), i_Q^G(\fS^{-1}\cP))\\
j^G(\cP_0) &\in Z(\End_{\fS^{-1}\fR[M]}(\fS^{-1}\cP)) \cong \fS^{-1}\fR
\end{align*}
where the centrality of $j^G(\cP_0)$ is from Proposition~\ref{prop:centrality}. We will describe a sense in which $J_{Q|P}(\cP_0)$ and $j^G(\cP_0)$ are ``universal.''

Let $B$ be a commutative Noetherian $\ZZ'$-algebra. By \cite[Lem B.13]{mt_theta}, $\cP_{B,0} = \cP_0\otimes_{\ZZ'}B$ is a projective generator of the depth $\leq r$ subcategory of smooth $B[M]$-modules and we have
$$\J_{Q|P}(\cP_{B,0}) = J_{Q|P}(\cP_0)\otimes_{\ZZ'}1_B$$ by Proposition~\ref{prop:compatibility}(ii).

Now let $\pi_0$ be a finitely generated $B[M]$-module of depth $\leq r$, with universal unramified twist $\pi := \pi_0\otimes_{\ZZ'} \Psi_M$. There is an integer $n$ and a surjection of $B[M]$-modules
$$\cP_{B,0}^{\oplus n}\to \pi_0.$$ By applying Proposition~\ref{prop:compatibility}(i) to each inclusion $\cP_{B,0}\to \cP_{B,0}^{\oplus n}$ we find that
$$J_{Q|P}(\cP_{B,0}^{\oplus n}) = \diag(J_{Q|P}(\cP_{B,0})).$$ Now applying Proposition~\ref{prop:compatibility}(i) to the surjection $\cP_{B,0}^{\oplus n}\to \pi_0,$ we find that $J_{Q|P}(\pi_0)$ is determined by the following recipe: given $v\in i_P^G(\fS^{-1}\pi)$, choose any preimage $v'$ of $v$ in the surjection $i_P^G(\fS^{-1}\cP_B^{\oplus n})\to i_P^G(\fS^{-1}\pi)$, then send $v'$ to $i_Q^G(\fS^{-1}\cP_B^{\oplus n})$ via $\diag(\J_{Q|P}(\cP_{B,0})) = \diag(J_{Q|P}(\cP_0)\otimes_{\ZZ'}1_B)$, then project back down along the surjection $i_Q^G(\fS^{-1}\cP_B^{\oplus n})\to i_Q^G(\fS^{-1}\pi)$.
\begin{thm}\label{thm:universalJ}
Let $B$ be a commutative Noetherian $\ZZ'$-algebra and let $\pi_0$ be any finitely generated smooth $B[M]$-module. The intertwining operator $J_{Q|P}(\pi_0)$ is determined by $J_{Q|P}(\cP_0)$ in the sense that the following diagram commutes for any resolution $\cP_{B,0}^{\oplus n}\twoheadrightarrow \pi_0$.
$$
\begin{tikzcd}
i_P^G (\fS^{-1}\cP_B) \arrow[rrr,"J_{Q|P}(\cP_0)\otimes_{\ZZ'}1_B"] \arrow[d,"\diag"]&&& i_{Q}^{G}(\fS^{-1}\cP_B) \arrow[d,"\diag"]\\
i_P^G (\fS^{-1}\cP_B^{\oplus n}) \arrow[rrr,"\diag(J_{Q|P}(\cP_0)\otimes_{\ZZ'}1_B)"] \arrow[d,twoheadrightarrow]&&& i_{Q}^{G}(\fS^{-1}\cP_B^{\oplus n}) \arrow[d,twoheadrightarrow]\\
i_P^G (\fS^{-1}\pi) \arrow[rrr,"J_{Q|P}(\pi_0)"] &&& i_{Q}^{G}(\fS^{-1}\pi)
\end{tikzcd}
$$
\end{thm}

Although $J_{Q|P}(\pi_0)$ has, up to this point, been defined only for finitely generated $\pi_0$, Theorem~\ref{thm:universalJ} allows us to extend the definition to non-finitely generated objects:
\begin{definition}
Let $\pi_0$ be an arbitrary smooth $B[M]$-module of depth $\leq r$ and choose any surjection $\oplus_I\cP_0\to \pi_0$. Define $J_{Q|P}(\pi_0)$ to be the morphism induced on $i_P^G(\fS^{-1}\pi)\to i_Q^G(\fS^{-1}\pi)$ by applying $J_{Q|P}(\cP_0)\otimes_{\ZZ'} 1_B$ diagonally to $i_P^G(\fS^{-1}\oplus_I\sigma_B)$  and passing to the quotient. We can define $j^G(\pi_0)$ similarly.
\end{definition}

Define $Z_0 := Z(\End_{B[M]}(\pi_0))$. The natural action of $\cZ_{\ZZ',M}$ on $\pi_0$ induces a homomorphism $\lambda_{\pi_0}:\cZ_{\ZZ',M}\to Z_0$. We adopt the notation of Proposition~\ref{prop:bernsteinscalarextension}: by extending scalars to $\ZZ'[M/M^0]$ and localizing we have a morphism
$$\widetilde{\lambda}_{\pi_0}:\fS^{-1}\fR\to {\cS}^{-1}Z,$$ where $Z:=Z_0[M/M^0] = Z(\End_{B[M/M^0][M]}(\pi))$ and $\cS$ is the multiplicative subset of $Z$ defined analogously to $\fS$. The following theorem states that $j^G(\cP_0)$ is the universal depth $\leq r$ $j$-function.

\begin{thm}\label{thm:universalj}
Suppose $\pi_0$ is a $B[M]$-module of depth $\leq r$ with $\cZ_{\ZZ',M}$ acting via $\lambda_{\pi_0}:\cZ_{\ZZ',M}\to Z_0$. Then  $$j^G(\pi_0) = \widetilde{\lambda}_{\pi_0}(j^G(\cP_0)),$$ where the equality takes place in the ring $\cS^{-1}Z$.
\end{thm}
\begin{proof}
Let us abbreviate $\lambda := \lambda_{\pi_0}$, $\cP_{\lambda,0} := \cP_0\otimes_{\cZ_{\ZZ',M},\lambda}Z_0$, and $\cP_{\lambda} := \cP \otimes_{\fR,\widetilde{\lambda}}Z$. 

Since $\pi_0$ is finitely generated as a $B[M]$-module, and $B\to \End_{B[M]}(\pi_0)$ factors through $Z_0$, we have that $\pi_0$ is finitely generated as a $Z_0[M]$-module. We could replace $Z_0$ with $B$ in what follows without losing generality. 

Fix a morphism $\phi:\cP_0\to \pi_0$, which induces $\phi_{Z_0}:\cP_{\lambda,0}\to \pi_0$ and $\widetilde{\phi}_{Z_0}:\cS^{-1}\cP_{\lambda}\to \cS^{-1}\pi$. Proposition~\ref{prop:compatibility_of_j}(i) gives $$j^G(\pi_0)|_{\widetilde{\phi}_{Z_0}(\cS^{-1}\cP_{\lambda})} = j^G(\phi(\cP_0)).$$ Note that this is an equality of scalars in $\cS^{-1}Z$.

From the compatibility of $J_{Q|P}(\sigma_0)$ with scalar extension (take $B$ in Proposition~\ref{prop:compatibility_of_j}(iii) to be $Z_0$)  we have $$\widetilde{\lambda}(j^G(\cP_0)) = j^G(\cP_{\lambda,0})\in \cS^{-1}Z.$$ Applying Proposition~\ref{prop:compatibility_of_j}(i) again gives us $j^G(\cP_{\lambda,0})=j^G(\phi(\cP_0))$ in $\cS^{-1}Z$. We conclude that the scalars $j^G(\cP_{\lambda,0})$ and $j^G(\pi_0)$ act the same on the subspace $\widetilde{\phi}_{Z_0}(\cS^{-1}\cP_{\lambda})$ of $\cS^{-1}\pi$. 

Letting $\phi$ vary over $\Hom_{\ZZ'[M]}(\cP_0,\pi_0)$, the subspaces $\widetilde{\phi}_{Z_0}(\cS^{-1}\cP_{\lambda})$ generate $\cS^{-1}\pi$ so $j^G(\cP_{\lambda,0})$ and $j^G(\pi_0)$ act the same on all of $\cS^{-1}\pi$, hence are identical elements of the ring $\cS^{-1}Z$.
\end{proof}

Note that if $\pi_0$ satisfies Schur's lemma, then $Z_0$ is simply $B$ in the above.

\section{Characterizing local Langlands in families: quasisplit classical groups}\label{sec:characterization}

\subsection{Langlands parameters}
Let $W_F$ denote the Weil group of $F$ and let $I_F$ and $P_F$ be the inertia and wild inertia subgroups, respectively.

For a connected reductive group defined over $F$, with $F$-points denoted by $G$, we consider its dual group $\widehat{G}$ as a $\ZZ[1/p]$-group scheme. Fix $^LG = \widehat{G}\rtimes W$ a finite form of the $L$-group, where $W$ is a finite quotient of $W_F$ through which the action of $W_F$ on $\widehat{G}$ factors. Here, $^LG$ is again considered as a $\ZZ[1/p]$-group scheme (possibly non-connected). Given a $\ZZ[1/p]$-algebra $R$ we say a homomorphism $\phi:W_F\to {^LG}(R)$ is an $L$-morphism if the composition $W_F\to {^LG}(R)\to W$ is equivalent to the natural projection $W_F\to W$. Note that specifying an $L$-morphism $\phi$ is equivalent to specifying a 1-cocyle $\phi^{\circ}:W_F\to \widehat{G}(R)$ in $Z^1(W_F,\widehat{G}(R))$.

Let $(P_F^e)_{e\in \mathbb{N}}$ be an exhaustive filtration of $P_F$ by open subgroups that are normal in $W_F$ with $P_F^0 = P_F$. Let $\Fr$ be a lift of a Frobenius element in $W_F/P_F$, let $s$ be a topological generator of the tame quotient $I_F/P_F$, and consider the discrete subgroup $(W_F/P_F)^0$ generated by $\Fr$ and $s$. By working with the preimage $W_F^0$ of this discrete subgroup in $W_F$, one defines the affine finite type $\ZZ[1/p]$-scheme $X_{{^LG}}^e = \Spec(\fR^e_{^LG})$ representing the functor
\begin{align*}
\underline{Z}^1(W_F^0/P_F^e,\widehat{G}) : \ZZ[1/p]\text{-algebras} &\to \text{Sets}\\
R&\mapsto Z^1(W_F^0/P_F^e,\widehat{G}(R))
\end{align*}

As explained in \cite[Rem 6.9(ii)]{dhkm_moduli} , if $\kappa$ is an algebraically closed field of characteristic zero, the $\kappa$-points of the affine GIT quotient scheme $X_{^LG}^e\sslash \widehat{G}$ are precisely the $L$-morphisms $\phi:W_F\to {^LG}(\kappa)$ that are trivial on $P_F^e$ and have the property that $\phi(W_F)$ is \emph{semisimple} (i.e., if it is contained in a proper parabolic subgroup, then it is contained in a Levi subgroup of that parabolic). We call such an $L$-morphism a \emph{semisimple parameter} over $\kappa$.

We define $\fR_{^LG} = \lim\limits_{\underset{e}{\leftarrow}}\fR_{^LG}^e$. The coarse moduli space defined by the corresponding ind-affine scheme was introduced for $GL_n$ over $W(\Fl)$ in \cite{helm_curtis} and expanded to $\ZZ[1/p]$ and other connected reductive groups independently in \cite{dhkm_moduli, zhu, fargues_scholze}.

Let $\Fr,s_1,\dots, s_k$ be a choice of generators of the finitely presented group $W_F^0/P_F^e$, where the $s_i$'s topologically generate $I_F/P_F^e$. We have ``universal elements'' $\Fam, \sigma_1,\dots, \sigma_k\in \widehat{G}(\fR_{^LG}^e)$ such that 
$$(\Fam,\Fr),(\sigma_1,s_1),\dots, (\sigma_k,s_k)\in {^LG}$$ satisfy the defining relations for $W_F^0/P_F^e$. The ``universal parameter'' $$\phi^e_{\univ}:W_F^0/P_F^e\to {^LG}(\fR_{^LG}^e)$$ is defined by $\phi^e_{\univ}(\Fr) = (\Fam,\Fr)$ and and $\phi^e_{\univ}(s_i) = ( \sigma_i,s_i)$. This is not truly a parameter since it is only defined on the discretized subgroup $W_F^0/P_F^e$. 

We recall the notion of $\ell$-adic continuity introduced in \cite[Sec 2.5]{dhkm_moduli}. Fix a prime $\ell\neq p$. 
\begin{definition}[{\cite[Def 2.12]{dhkm_moduli}}]\label{def:ladiccontinuity}
Let $R$ be a $\ZZ[1/p]$-algebra. We say an $L$-morphism $\phi:W_F\to {^LG}(R)$ is $\ell$-adically continuous if there exists a ring homomorphism $f:R'\to R$ with $R'$ $\ell$-adically separated and $\phi':W_F\to {^LG(R')}$ satisfying:
\begin{enumerate}[(i)]
\item $\phi = {^LG}(f)\circ \phi'$
\item ${^LG}(\pi_n)\circ \phi':W_F\to {^LG}(R'/\ell^n R')$ is continuous for all $n$ where $\pi_n$ denotes the projection $R'\to R'/\ell^nR'$.
\end{enumerate}
\end{definition}

By \cite[Thm 2.13, Thm 4.1]{dhkm_moduli}, if we extend scalars to $\fR_{^LG,\ell}^e = \fR_{^LG}^e\otimes_{\ZZ[1/p]}\ZZ_{\ell}$, the ring $\fR_{^LG,\ell}^e$ is $\ell$-adically separated and there is a unique $\ell$-adically continuous $L$-morphism $\phi^e_{\ell,\univ}:W_F/P_F^e\to {^LG}(\fR_{^LG,\ell}^e)$ extending $\phi^e_{\univ}$. Moreover, $\fR_{^LG,\ell}^e$ is independent of the choice of topological generators, up to canonical isomorphism, and every $\ell$-adically continuous $L$-morphism $\phi:W_F/P_F^e\to {^LG}(R)$ over a Noetherian $\ell$-adically separated $\ZZ_{\ell}$-algebra $R$ arises by base changing $\phi^e_{\ell,\univ}$ along a homomorphism $\fR^e_{^LG,\ell}\to R$. 

Let $\kappa$ be an algebraically closed field of characteristic zero that is a $\ZZ_{\ell}$-algebra. We say a semisimple parameter $W_F\to {^LG(\kappa)}$ is $\ell$-adically continuous if it is continuous in the $\ell$-adic topology on $\kappa$. For example, when $\kappa = \qlb$ and $\phi$ is \emph{integral}, i.e., $\phi(\Fr)$ is a compact element, this coincides with Definition~\ref{def:ladiccontinuity} by \cite[Prop 6.23]{dhkm_conjecture}.

\subsection{Classical groups and conjugate self-dual parameters}\label{classicalgroupsdefinitions}
Let $E/F$ be a trivial or a quadratic extension of $p$-adic fields, let $c$ be the generator of $\Gal(E/F)$, let $\epsilon\in \{\pm 1\}$ and let $(V,h)$ be an $\epsilon$-hermitian space. We define the isometry group
$$U(V,h) = \{g\in \mathrm{GL}_E(V)\ :\ h(gv, gw) = h(v,w),\ v,w\in V\}.$$
\begin{itemize}
\item When $E/F$ is quadratic, $U(V,h)$ is a unitary group,
\item When $E=F$ and $\epsilon = -1$, $U(V,h)$ is a symplectic group,
\item When $E=F$ and $\epsilon = 1$, $U(V,h)$ is an orthogonal group.
\end{itemize}
Note that $U(V,h)$ is the $F$-points of a possibly disconnected reductive group defined over $F$. In this article we use the following definition of ``classical group:''

\begin{definition}
A classical group $G$ is a unitary, symplectic, or odd special orthogonal group. If $V$ has dimension zero, $G$ is the trivial group. In particular, $G$ is connected. 
\end{definition}

Let $G$ be a quasisplit classical group. We define conjugate self-dual Langlands parameters for $G$ following \cite[Section 8]{ggp_symplectic_local}. 

First, suppose $E=F$ and $G$ is a symplectic or odd special orthogonal group. In these cases, $G$ is split. Let ${^LG}=\widehat{G}$ be its Langlands dual group, defined over $\ZZ$ by a Chevalley basis. Let $m= 2n+1$ or $2n$, respectively, and let $\cR:{^LG}\to \GL_m$ be the standard representation. If $A$ is a $\ZZ[1/p]$-algebra and $\phi:W_F\to {^LG}(A)$ is a semisimple $A$-valued Langlands parameter we will consider the composite $\cR\circ\phi:W_F\to \GL_m(A)$ and call it a self-dual $A$-valued parameter for $G$.

Next, suppose $G$ is unitary, let $^LG = \widehat{G}\rtimes \Gal(E/F) = \GL_m\rtimes \Gal(E/F)$, and let $\phi:W_F\to {^LG}(A)$ be an $A$-valued Langlands parameter. We will let $\cR\circ\phi$ denote the restriction $\phi|_{W_E}:W_E\to \GL_m(A)$ and refer to it as a conjugate self-dual $A$-valued parameter for $G$.

\begin{prop}[{\cite[Thm 8.1]{ggp_symplectic_local}}]
Let $\phi$ and $\phi'$ be parameters valued in ${^LG}(\C)$. Then $\phi$ and $\phi'$ are conjugate in $\widehat{G}(\C)$ if and only if $\cR\circ\phi$ and $\cR\circ \phi'$ are conjugate in $\GL_m(\C)$.
\end{prop}

\noindent To shorten notation we will often abbreviate $\cR\circ \phi$ by $\cR\phi$.

\subsection{Harish-Chandra $\mu$-function}\label{subsec:mu}
Let $G$ be a classical group coming from an $\epsilon$-hermitian space $V$ as above. Let $V'$ be an $\epsilon$-hermitian space of the form $V' = W\oplus V \oplus W'$ where $W$, $W'$ are dimension $m$ and totally isotropic and $V'$ is orthogonal to $W\oplus W'$. This specifies a parabolic subgroup $P$ of the classical group $G'$ associated to $V'$, where $G'$ has the same type as $G$, having Levi subgroup $M = G\times GL(W) = G\times GL_m(F)$.

Fix a nontrivial additive character $\psi:F\to W(\Fl)^{\times}$. Let $U$ be the unipotent radical of $P$ and choose Haar measures $du$ and $du'$ on $U$ and $\overline{U}$ respectively, normalized with respect to $\psi$ according to the procedure described in \cite[B.2]{gan_ichino_2014}.

For a fixed depth $r\geq 0$ let $\cP(G)_r$ denote the progenerator of the depth $\leq r$ subcategory of $\Rep_{W(\Fl)}(G)$. Denote the Harish-Chandra $j$ function with respect to $\psi$ as $$j^{G'}_{\psi}(\cP(G)_r\otimes \cP(GL_m)_{r'})$$ as in Section~\ref{sec:jfunction}, where the Haar measures $du$ and $du'$ are normalized with respect to $\psi$. We note that $j_{\psi}^{G'}(\cP(G)_r\otimes \cP(GL_m)_{r'})$ is universal in the sense of Theorem~\ref{thm:universalj}. 

Note that $M$ is a maximal Levi subgroup of $G'$, and sending $(g_V,g_W)\in M$ to $\text{val}_E(\det(g_W))$ induces an isomorphism $W(\Fl)[M/M^0]\cong W(\Fl)[GL_m/GL_m^0]\cong W(\Fl)[X^{\pm 1}]$. In this sense, we view $ j_{\psi}^{G'}(\cP(G)_r\otimes \cP(GL_m)_{r'})$ as an element of $\fS^{-1}(\cZ_{W(\Fl),M}[X^{\pm 1}])$ where $\fS$ is the multiplicative subset of $\cZ_{W(\Fl),M}[X^{\pm 1}]$ consisting of polynomials whose first and last coefficients are units.

\subsection{Gamma factors for $L$-parameters}
Consider a triple $(R,\rho,\psi)$, where $R$ is a Noetherian $\ell$-adically separated $\wflb$-algebra, $\psi:F\to \wflb^{\times}$ is a nontrivial character, and $\rho:W_F\to \GL_n(R)$ is an $\ell$-adically continuous homomorphism. Let $S$ be the multiplicative subset of $R[X^{\pm 1}]$ consisting of polynomials whose first and last coefficients are units. In \cite{HM_deligne_langlands} there is associated to any such triple an element $$\gamma(\rho, X, \psi)\in S^{-1}\big(R[X,X^{-1}]\big)^{\times}.$$ This construction is compatible with extension of scalars along any homomorphism $R\to R'$, and it specializes to the classical Deligne--Langlands gamma factor when $R = \kappa$ is an algebraically closed field of characteristic zero.

Let $\phi:W_F\to {^LG}(A)$ and $\tau:W_F\to \GL_m(B)$ be semisimple parameters valued in Noetherian $W(\Fl)$-algebras $A$ and $B$, respectively. We take $R = A\otimes B$ and consider $$\gamma(\cR\phi\otimes \tau, r, X, \psi)\in S^{-1}\left(R[X,X^{-1}]\right)^{\times}.$$ Given homomorphisms $f:A\to A'$ and $g:B\to B'$ we have, by compatibility with scalar extension,
$$(f\otimes g)\left(\gamma(\cR\phi\otimes \tau, X, \psi)\right) = \gamma((\cR\phi\otimes\tau)\otimes_{f\otimes g}(A'\otimes B'),X, \psi).$$

\subsection{Converse theorem for classical groups}\label{subsec:conversethm}

Given $\phi:W_F\to {^LG}(A)$ and $\phi':W_F\to GL_m(B)$ as above, we define $j_{\psi}(\phi\otimes\phi',X,\psi)$ to be
$$\left(\gamma(\cR\phi\otimes {\phi'}^{\vee},X,\psi_E)\gamma((\cR\phi)^{\vee}\otimes \phi',X^{-1},\psi_E^{-1})\gamma(R\circ\phi',X^2,\psi)\gamma(R\circ{\phi'}^{\vee},X^{-2},\psi^{-1})\right)^{-1},$$
where $\cR$ is the standard representation described above and $R$ is $\text{Sym}^2$, $\wedge^2$, $\text{As}^+$, or $\text{As}^-$ if $G$ is orthogonal, symplectic, even unitary, or odd unitary, respectively. For constructing the Plancherel measure, it is customary to work with $\mu_{\psi}(\phi\otimes\phi',X,\psi) = j_{\psi}(\phi\otimes\phi',X,\psi)^{-1}$, but we will stay with $j_\psi(\phi\otimes\phi',X,\psi)$ in parallel with the framework of Section~\ref{sec:jfunction}.

\begin{prop}[\cite{dhkm_conjecture} Prop 7.7]\label{charzeroconverse}
Let $G$ be a classical group and let $\Kbar$ be an algebraic closure of $\K = \Frac(W(\Fl))$. Suppose $\phi_1, \phi_2:W_F\to {^LG}(\Kbar)$ are two semisimple parameters such that 
$$j_{\psi}(\phi_1\otimes \tau,X,\psi) = j_{\psi}(\phi_2\otimes \tau,X,\psi)$$ for all irreducible semisimple parameters $\tau:W_F\to GL_t(\Kbar)$ for all $t\leq m$. Then $\phi_1$ and $\phi_2$ are $\widehat{G}(\Kbar)$-conjugate.
\end{prop}

\begin{proof} Note that the equality $j_{\psi}(\phi_1\otimes \tau,X,\psi) = j_{\psi}(\phi_2\otimes \tau,X,\psi)$ in the statement of Proposition~\ref{charzeroconverse} implies the hypothesis of \cite[Prop 7.7]{dhkm_conjecture} because for $\mu_{\psi} = j_{\psi}^{-1}$ and the factors $$\gamma(R\circ\tau,X^2,\psi)\gamma(R\circ{\tau}^{\vee},X^{-2},\psi^{-1})$$ depend only on $\tau$ and can be canceled.
The statement proved in \cite{dhkm_conjecture} is for $\C$-valued parameters, but $\Kbar$ is an algebraically closed field of characteristic zero that has the same cardinality as $\C$ and thus is isomorphic to $\C$.
\end{proof}
We record the following reformulation of the converse theorem:
\begin{corollary}\label{cor:stronger_converse}
The same statement as Proposition~\ref{charzeroconverse} remains true if $j_{\psi}(\phi_1\otimes \tau,X,\psi) = j_{\psi}(\phi_2\otimes \tau,X,\psi)$ for all semisimple parameters $\tau:W_F\to GL_m(\Kbar)$.
\end{corollary}
\begin{proof}
For any $t<m$, consider the semisimple $W_F$ representation $\tau$ given by $\tau_0\oplus \chi^{\oplus m-t}$ where $\tau$ is irreducible and $\chi$ is a character. By the multiplicativity of gamma factors we write $j_{\psi}(\phi_i\otimes \tau,X,\psi)$ as $j_{\psi}(\phi_i\otimes\tau_0,X,\psi)j_{\psi}(\phi_i\otimes\chi,X,\psi)^{m-t}$. By the stability of gamma factors, if $\chi$ is highly ramified, $\gamma(\cR\phi_1\otimes\chi,X,\psi) = \gamma(\cR\phi_2\otimes\chi,X,\psi)$, so after canceling factors on either side of the equality $j_{\psi}(\phi_1\otimes \tau,X,\psi) = j_{\psi}(\phi_2\otimes \tau,X,\psi)$, we have $j_{\psi}(\phi_1\otimes\tau_0,X\psi) = j_{\psi}(\phi_2\otimes\tau_0,X,\psi)$. 
\end{proof}

\subsection{Characterizing a local Langlands morphism with universal factors}
Now suppose we are given a continuous ring homomorphism
$$\mathcal{L}:\fR_{^LG, \ZZ_{\ell}'}^{\widehat{G}}\to \cZ_{G, \ZZ_{\ell}'}$$ as in the introduction. We will show that, in the case of classical groups, there exists at most one ring homomorphism $\mathcal{L}$ taking the Langlands--Deligne gamma factors of a tensor product to the universal Plancherel measure of a maximal Levi subgroup. First, we prove the result over $\Kbar$, an algebraic closure of the fraction field $\K$ of $W(\Fl)$. Fix $\psi:F\to W(\Fl)^{\times}$. Identify $G\times GL_m$ with a Levi subgroup in a classical group $G'$ as in Subsection~\ref{subsec:mu}. 

\begin{thm}\label{uniquenessqlb}
There exists at most one ring homomorphism $$\mathcal{L}_{\Kbar}:\fR_{^LG, \Kbar}^{\widehat{G}}\to \cZ_{G,\Kbar}$$ such that, for $r\geq 0$, $$(\mathcal{L}_{\Kbar}\otimes \mathcal{L}_{GL_m,\Kbar})(j_{\psi}(\phi_{\univ,\Kbar}^{e(r)}\otimes \phi_{\univ,\Kbar}^{e(r')},X,R,\psi)) = j_{\psi}^{G'}(\cP(G)_{\Kbar,0,r}\otimes \cP(GL_m)_{\Kbar,0,r'}).$$
\end{thm}
\begin{proof}
In this proof only, we will drop the subscripts $\Kbar$, ${^LG}$, and $G$ to ease notation. Let $\cL_1$ and $\cL_2$ be two such homomorphisms. Fix $r\geq 0$ and let $e_1 = e_1(r)$ and $e_2 = e_2(r)$ denote the integers $e(r)$ defined as above for $\cL_1$ and $\cL_2$, respectively. Each $\cL_i$ restricts to finite depth:
$$\cL_i:(\fR^{e_i(r)})^{\hat{G}}\to \cZ_r\ .$$ We will show that $e_1=e_2$ and that these two restrictions coincide.

Consider a point $x:\cZ_{r}\to \Kbar$, and let $y_i:(\fR^{e_i})^{\hat{G}}\to \Kbar$ be its pullback under $\cL_i$. Each $y_i$ determines a closed point of the GIT quotient $Z^1(W_F, \widehat{G}_{\Kbar})\sslash \widehat{G}(\Kbar)$, hence a closed $\widehat{G}(\Kbar)$-orbit of semisimple 1-cocycles $W_F\to \widehat{G}(\Kbar)$. Choose one such cocycle $\phi_i:W_F\to \widehat{G}(\Kbar)$; the argument that follows does not depend on this choice. By the universal property, there exist unique homomorphisms 
$$f_i:\fR^{e_i}\to \Kbar$$ such that $$\phi_i = \phi^{e_i,\univ}\otimes_{\fR^{e_i},f_i}\Kbar.$$ Note that $f_i|_{(\fR^{e_i})^{\hat{G}}} = y_i$.

For every nonnegative $t$ and every $f:\fR^{e(r')}_{GL_m}\to \Kbar$ we set $\phi = \phi^{e(r'),\univ}_{GL_m}\otimes_{\fR^{e(r')}_{GL_m},f}\Kbar$. Set $y' = f|_{(\fR^{e(r')}_{GL_m})^{GL_m(\Kbar)}}$ and let $x' = y'\circ \cL_{GL_m}^{-1}:\cZ_{GL_m}\to \Kbar$.

Note that $j_{\psi}(\phi_i\otimes \phi',X,R,\psi)$ depends only on the semisimplification of $\phi_i$ and $\phi'$, and the coefficients of $j_{\psi}(\phi_{^LG}^{e,\univ}\otimes \phi_{GL_m}^{e(r'),\univ},X,R,\psi)$ live in the subring $(\fR^e_{^LG})^{\widehat{G}}\otimes(\fR^{e(r')}_{GL_m})^{GL_m(\Kbar)}$.

\begin{align*}
j_{\psi}(\phi_i\otimes \phi',X,R,\psi) &= (f_i\otimes f)\left(j_{\psi}(\phi_{^LG}^{e,\univ}\otimes \phi_{GL_m}^{e(r'),\univ},X,R,\psi))\right) \\
& = (y_i\otimes y)\left(j_{\psi}(\phi_{^LG}^{e,univ}\otimes \phi_{GL_m}^{e(r'),\univ},X,R,\psi))\right)\\
& = ((x\circ \cL_i)\otimes (x'\circ \cL_{GL_m}))\left(j_{\psi}(\phi_{^LG}^{e,\univ}\otimes \phi_{GL_m}^{e(r'),univ},X,R,\psi))\right)\\
& = (x\otimes x')\left(j_{\psi}(\cP(G)_{\Kbar,0,r}\otimes \cP(GL_m)_{\Kbar,0,r'}\right).
\end{align*}
The final expression is the same whether $i$ is $1$ or $2$, so we conclude
$$j_{\psi}(\phi_1\otimes \phi',X,R,\psi)=j_{\psi}(\phi_2\otimes \phi',X,R,\psi)$$ for all $\phi'$.

By Corollary~\ref{cor:stronger_converse} this implies $\phi_1$ and $\phi_2$ are in the same $\widehat{G}(\Kbar)$ orbit of $Z^1(W_F,\widehat{G}_{\Kbar})$. In particular, $e_1=e_2$; set $e=e_1=e_2$. Now $f_1$ coincides with $f_2$ on the subring of invariants $(\fR^e_{\Kbar})^{\widehat{G}(\Kbar)}$. In other words, $y_1 = y_2$ and
$$x\circ \cL_1 = x\circ \cL_2.$$ Now let $\alpha$ be any element in $(\fR^e)^{\widehat{G}}$. We have shown that $\cL_1(\alpha)- \cL_2(\alpha)$ is in the kernel of every homomorphism $x:\cZ_r\to \Kbar$. 

Given a homomorphism $x:\cZ_r\to \kappa$ where $\kappa$ is an arbitrary field of characteristic zero, Zariski's lemma implies $x$ factors through a subfield $\kappa'$ that is isomorphic to $\Kbar$, so $\cL_1(\alpha)- \cL_2(\alpha)$ is zero at $x$ by the above argument. Since $\cZ_r$ has no $\ell$-torsion, characteristic zero points $x:\cZ_r\to \kappa$ form a dense subset of $\Spec(\cZ_r)$ and, as $\cZ_r$ is reduced, this implies $\cL_1(\alpha) -\cL_2(\alpha)=0$, as desired.
\end{proof}

Since we have descended $j_{\psi}$ to $\ZZ'$ integrally we can now state the following corollary, which characterizes local Langlands morphisms over $\ZZ_{\ell}' = \ZZ_{\ell}[\sqrt{q}]$.

\begin{corollary}
In the setup of Theorem~\ref{uniquenessqlb}, there exists at most one continuous homomorphism
$$\mathcal{L}:\fR_{^LG, \ZZ_{\ell}'}^{\widehat{G}}\to \cZ_{G, \ZZ_{\ell}'}$$ such that $$(\mathcal{L}\otimes \mathcal{L}_{GL_m})(j_{\psi}(\phi_{\univ}^{e(r)}\otimes \phi_{\univ}^{e(r')},X,R,\psi)) = j_{\psi}^{G'}(\cP(G)_r\otimes \cP(GL_m)_{r'}).$$
\end{corollary}
\begin{proof}
Both the source and target are compatible with flat base change, so after extending scalars along $\ZZ_{\ell}'\to \Kbar$, we can apply Theorem~\ref{uniquenessqlb}.
\end{proof}

\bibliographystyle{alpha}
\bibliography{lesrefer}
\end{document}